\documentclass[12pt]{amsart}
\usepackage{latexsym,epsfig,amssymb,amsmath,amsthm,color,url,tikz}
\usetikzlibrary{arrows,positioning} 
\usepackage{graphicx}
\usepackage{inputenc}
\usepackage[square,comma,sort&compress]{natbib}
\setcitestyle{numbers,multirow,import}
\usepackage{graphicx}
\allowdisplaybreaks
\numberwithin{equation}{section}
\numberwithin{figure}{section}

\newtheorem{Theorem}{Theorem}[section]
\newtheorem{Lemma}[Theorem]{Lemma}
\newtheorem{Remark}[Theorem]{Remark}
\newtheorem{Example}[Theorem]{Example}
\newtheorem{Proposition}[Theorem]{Proposition}
\newtheorem{Definition}[Theorem]{Definition}
\newtheorem{Corollary}[Theorem]{Corollary}

\def\bZ{\boldsymbol Z}

\def\be{\boldsymbol e}

\def\bz{\boldsymbol z}

\def\min{\wedge}

\newcommand{\deltain}{\delta_{\text in}}
\newcommand{\deltaout}{\delta_{\text out}}

\newcommand{\bthe}{\begin{Theorem}}
\newcommand{\ethe}{\end{Theorem}}

\newcommand{\ble}{\begin{Lemma}}
\newcommand{\ele}{\end{Lemma}}

\newcommand{\bde}{\begin{Definition}}
\newcommand{\ede}{\end{Definition}}

\newcommand{\bco}{\begin{Corollary}}
\newcommand{\eco}{\end{Corollary}}

\newcommand{\bpr}{\begin{Proposition}}
\newcommand{\epr}{\end{Proposition}}

\newcommand{\brem}{\begin{Remark}}
\newcommand{\erem}{\end{Remark}}

\newcommand{\bexam}{\begin{Example}}
\newcommand{\eexam}{\end{Example}}

\newcommand{\beqq}{\begin{equation}}
\newcommand{\eeqq}{\end{equation}}

\newcommand{\beao}{\begin{eqnarray*}}
\newcommand{\eeao}{\end{eqnarray*}\noindent}

\newcommand{\beam}{\begin{eqnarray}}
\newcommand{\eeam}{\end{eqnarray}\noindent}

\newcommand{\barr}{\begin{array}}
\newcommand{\earr}{\end{array}}

\newcommand{\bproof}{\begin{proof}}
\newcommand{\eproof}{\end{proof}}

\newcommand{\sid}[1]{{\color{black} #1}}

\newcommand{\Din}{D^{\text{in}}}
\newcommand{\Dout}{D^{\text{out}}}

\newcommand{\cin}{c_\text{in}}
\newcommand{\cout}{c_\text{out}}

\newcommand{\Nin}{N^{\text{in}}}
\newcommand{\Nout}{N^{\text{out}}}

\newcommand{\pin}{p^{\text{in}}}
\newcommand{\pout}{p^{\text{out}}}

\newcommand{\ain}{\iota_\text{in}}
\newcommand{\aout}{\iota_\text{out}}

\newcommand{\dd}{\mathrm{d}}
\newcommand{\PP}{\textbf{P}}
\newcommand{\EE}{\textbf{E}}
\newcommand{\ind}{\textbf{1}}

\newcommand{\convp}{\stackrel{P}{\longrightarrow}}
\newcommand{\convw}{\Rightarrow}
\newcommand{\convas}{\stackrel{\text{a.s.}}{\longrightarrow}}
\newcommand{\convv}{\stackrel{v}{\longrightarrow}}

\newcommand{\RR}{\mathbb{R}}

\newcommand{\bfD}{\boldsymbol{\mathcal{D}}}
\newcommand{\bd}{\boldsymbol{d}}

\newcommand{\origin}{\boldsymbol{0}}

\newcommand{\tw}[1]{{\color{black} #1}}
\newcommand{\Ia}{I^{(0)}}
\newcommand{\Ib}{I^{(1)}}
\newcommand{\Oa}{O^{(0)}}
\newcommand{\Ob}{O^{(1)}}
\newcommand{\Ebar}{\overline{E}}
\newcommand{\calI}{\mathcal{I}}
\newcommand{\calO}{\mathcal{O}}

\begin{document}
\bibliographystyle{plain}

\title[Degree Growth in Directed Networks]{Degree Growth Rates and Index Estimation in a Directed Preferential Attachment Model}

\author{Tiandong Wang and Sidney I. Resnick} \thanks{This work was supported by Army MURI grant W911NF-12-1-0385 to Cornell University.}

\begin{abstract}
Preferential attachment is widely used to model power-law behavior
of degree distributions in both directed and undirected networks.
In a directed preferential attachment model, despite the well-known marginal power-law degree distributions,
not much investigation has been done on the joint behavior of the in-  and out-degree growth.
Also, statistical estimates of the marginal tail exponent of the power-law degree distribution often use
the Hill estimator as one of the key summary statistics, even though no theoretical justification has been given.
This paper focuses on convergence of the joint empirical measure for in- and out-degrees and proves the
consistency of the Hill estimator. To do this, we first derive the asymptotic behavior of the joint degree sequences
by embedding the in- and out-degrees of a fixed node into a pair of switched birth processes with immigration and
then establish the convergence of the joint tail empirical measure. 
From these steps, the consistency of the Hill estimators is obtained.
\end{abstract}
\maketitle

\bigskip\noindent{\bf MSC Classes:} 60G70, 60B10, 60G55, 60G57, 05C80, 62E20.\smallskip

\section{Introduction.}
The preferential attachment model generates a growing sequence of random graphs based on the assumption that
popular nodes with large degrees attract more edges. Nodes and edges are added to the graph following probabilistic rules.
Such mechanism provides a basis for studying the evolution of social networks, collaborator and
citation networks, as well as recommender networks, and is applicable to both directed and undirected graphs.
Mathematical formulations of the undirected preferential attachment model are available in \cite{bhamidi:2007,durrett:2010b,vanderHofstad:2017},
and those of the directed model can be found in \cite{krapivsky:redner:2001,bollobas:borgs:chayes:riordan:2003}.
This paper only considers the directed model where at each stage, a new node is born and either it points to one of the existing nodes or one of the
existing nodes attaches to the new node.
Results on the degree growth in the undirected case are investigated in \cite{athreya:ghosh:sethuraman:2008,wang:resnick:2017}.

Empirical studies on social network data often reveal that in- and out-degree distributions marginally follow power laws.
Theoretically, this is also true for \emph{linear} preferential attachment models, which makes preferential attachment appealing in network modeling;
see \cite{bollobas:borgs:chayes:riordan:2003,krapivsky:2001,krapivsky:redner:2001} for references.
Also, the empirical joint degree frequency 
 converges to the probability mass function (pmf) of a pair of limit random variables that are jointly regularly varying 
(cf. \cite{krapivsky:redner:2001,wang:resnick:2016,resnick:samorodnitsky:towsley:davis:willis:wan:2016,resnick:samorodnitsky:2015}).
However, questions related to joint degree growth and index estimation still remain unresolved. In this paper, we focus on three main problems:
\begin{enumerate}
\item \sid{For a fixed node} in a linear preferential attachment graph, what is the joint
  behavior of in- and out-degree as the graph size grows?
\item 
\sid{What are the} convergence properties of the tail empirical joint
measure of in- and out-degrees indexed by node?
\item  \sid{When} estimating the marginal power-law indices of in- and
  out-degree, can we use the Hill estimator as a consistent estimator?
\end{enumerate}

What is the justification for interest in Hill estimation of power-law
indices for network data?
Repositories of large network datasets such 
as KONECT (http://konect.uni-koblenz.de/, \cite{kunegis:2013}) 
provide summary statistics for all the archived network datasets and
among the summary statistics are estimates of degree indices computed
with Hill estimators, despite the fact that evidence for Hill estimator consistency
is scant for network data \cite{wang:resnick:2017}.

Another justification is robust parameter estimation methods in
network models based on extreme value techniques.
In \cite{wan:wang:davis:resnick:2017b}, we couple the Hill estimation
of marginal degree distribution tail indices with a minimum distance threshold
selection method introduced in \cite{clauset:shalizi:newman:2009}  
and compare this method with the parametric estimation approaches used
in \cite{wan:wang:davis:resnick:2017}. The Hill estimation is more
robust against 
modeling errors and data corruptions. Therefore, an affirmative answer to the third question helps justify all of these inference methodologies.

In the directed case, consistency of the two marginal Hill estimators 
results from resolving 
the first two questions,
since in a similar vein to \cite{wang:resnick:2017}, we consider the
Hill estimator as a functional of the marginal tail empirical
measure. 
So convergence results of marginal tail empirical measures lead to the
consistency of Hill estimators by a mapping argument.

To answer the first question \sid{about degree behavior of fixed nodes
as graph size grows,} we mimic in- and out-degree growth of a fixed
node using pairs of \emph{switched birth processes with immigration}
(SBI processes). The SBI processes use Bernoulli switching between
pairs of independent \emph{birth processes with immigration} (BI
processes).  We embed the directed network growth model into a
sequence of paired SBI processes.  Whenever a new node is added to the
network, a new pair of SBI processes is initiated.  Using convergence
results for BI processes (cf. \cite[Chapter 5.11]{resnick:1992},
\cite{tavare:1987, wang:resnick:2017}), we give the joint limits of
the in- and out-degrees of a fixed node as well as the joint maximal
degree growth.  Proving the convergence of the tail empirical joint
measure in the second question requires showing concentration results
for degree counts compared with expected degree counts.  With
embedding techniques, we  prove the 
limit distribution of the empirical joint degree frequencies in a way
that is different from the one used in
\cite{resnick:samorodnitsky:towsley:davis:willis:wan:2016}, and then
justify the concentration results.

Our paper is structured as follows. In the rest of this section, we review background on the tail empirical measure
and Hill estimator. Section~\ref{sec:motiv} sets up the linear preferential attachment model
and formulates the power-law phenomena in network degree distributions.
Section~\ref{sec:SBI} summarizes facts about BI processes and introduces the SBI process, which is the foundation
of the embedding technique. We analyze the joint in- and out-degree growth in Section~\ref{sec:embed} by embedding
it into a sequence of paired SBI processes and derive convergence results of the in- and out-degrees for a fixed node.
Results on the convergence of the joint empirical measure are given in Section~\ref{sec:conv} and the consistency of
Hill estimators for both in- and out-degrees is proved in Section~\ref{sec:Hill}.
Useful concentration results are collected in Section~\ref{sec:conc}.

\subsection{Background}
Our approach to the Hill estimator considers it as a functional
of the tail empirical measure so we start with necessary background
and review  standard results (cf. \cite[Chapter 3.3.5 and 6.1.4]{resnickbook:2007}).  

\subsubsection{Non-standard regular variation.}
Let $M_+([0,\infty]^2\setminus \{\boldsymbol 0\})$ be the set of Radon measures on $[0,\infty]^2_+\setminus \{\boldsymbol 0\}$.
Then a random vector $(X,Y)$ is {\it non-standard regularly varying} on
$[0,\infty]^2_+\setminus \{\boldsymbol 0\}$ if there exist scaling functions $b_i(t)\to\infty$, $i=1,2$ such
 that as $t\to\infty$,
\beqq\label{eq:MRV}
t\PP\left[\left(\frac{X}{b_1(t)},\frac{Y}{b_2(t)}\right)\in\cdot\right]\convv \nu(\cdot),\quad \mbox{in }M_+([0,\infty]^2\setminus \{\boldsymbol 0\}),
\eeqq
where  $\nu(\cdot) \in M_+([0,\infty]^2\setminus \{\boldsymbol 0\})$
is called the limit or tail measure
\cite{resnick:samorodnitsky:2015, resnick:samorodnitsky:towsley:davis:willis:wan:2016}, and ``$\convv$" denotes the vague convergence of measures
in $M_+([0,\infty]^2\setminus \{\boldsymbol 0\})$.
The phrasing in \eqref{eq:MRV} implies the marginal distributions have regularly varying tails.

\subsubsection{Hill Estimator}
\sid{For $x\in (0,\infty]$, define the measure
$\epsilon_x(\cdot)$ on Borel subsets $A$ of $(0,\infty]$ by }
\[
\epsilon_x(A)=
\begin{cases}
1\qquad x\in A,\\
0\qquad x\notin A,
\end{cases}
\quad \text{for }A\in\mathcal{E}. 
\]
Let $M_+((0,\infty])$  be the set
of  non-negative Radon measures on $(0,\infty]$. 
A point measure $m$ is an element of $M_+((0,\infty])$ of the form
\beqq\label{eq:pm}
m = \sum_i\epsilon_{x_i}.
\eeqq

For $\{X_n, n\geq 1\}$ iid  and non-negative with common
regularly varying distribution
tail $\overline{F}\in RV_{-\iota}$, $\iota>0$, there  
exists a sequence $\{b(n)\}$ \sid{satisfying $P[X_1>b(n)]\sim 1/n$,}
such that  
for any $k_n\to\infty$, $k_n/n\to 0$, 
\beqq\label{eq:bnkX}
\frac{1}{k_n}\sum_{i=1}^n\epsilon_{X_i/b(n/k_n)}\Rightarrow \nu_\iota,\quad\text{ in }M_+((0,\infty]),
\eeqq
where the limit measure $\nu_\iota$ satisfies $\nu_\iota(y,\infty] =y^{-\iota}$, $y>0$.

Define the Hill estimator $H_{k,n}$ based on $k$ upper order statistics of
$\{X_1,\dots,X_n\}$ as \cite{hill:1975}
\beqq\label{eq:hill_def}
H_{k,n} := \frac{1}{k}\sum_{i=1}^{k}\log\frac{X_{(i)}}{X_{(k+1)}},
\eeqq
where $X_{(1)}\ge X_{(2)}\ge \ldots\ge X_{(n)}$ are order statistics of $\{X_i:1\le i\le n\}$.
In the iid case there are many proofs of consistency \cite{mason:1982,mason:turova:1994,hall:1982,dehaan:resnick:1998,
  csorgo:haeusler:mason:1991a}:  For
$k=k_n\to\infty,\,k_n/n\to 0$, we have
\begin{equation}\label{e:Hillconv}
H_{k_n,n} \convp 1/\iota\qquad\text{as }n\to\infty.
\end{equation}
The treatment in \cite[Theorem~4.2]{resnickbook:2007} approaches 
consistency by showing \eqref{e:Hillconv} follows from \eqref{eq:bnkX}
and we follow this approach for the network context where the iid case
is inapplicable.

\subsubsection{Node degrees.}
The next section constructs a directed preferential attachment model,
and gives behavior of $\bigl(\Din_v(n),\Dout_v(n)\bigr)$, the in- and
out-degrees of node 
$v$ at the $n$th stage of construction.  \sid{These degrees when scaled by
appropriate powers of $n$ (see \eqref{eq:conv_pair}) have limits and}
Theorem~\ref{thm:tail_meas}
shows that 
the degree sequences $\bigl(\Din_v(n), \Dout_v(n)\bigr)_{1\le v\le n}$
have a joint \sid{tail}
empirical measure
\beqq\label{emp_meas}
\frac{1}{k_n}\sum_{v} \epsilon_{\bigl(\Din_v(n)/b_1(n/k_n), \Dout_v(n)/b_2(n/k_n)\bigr)}
\eeqq
that converges weakly to some limit measure in $M_+([0,\infty]^2\setminus\{\boldsymbol{0}\})$,
\sid{where $b_1(n),b_2(n)$ are appropriate power law scaling functions and}
$k_n$ is some intermediate sequence such that
\[
k_n/n\to 0,\qquad k_n\to\infty,\quad \text{as }n\to\infty.
\]
It also follows from \eqref{emp_meas} that for some 
tail indices $\ain$, $\aout$, and intermediate sequence $k_n$,
\begin{align}
\frac{1}{k_n}\sum_{v}\epsilon_{\Din_v(n)/b_1(n/k_n)}&\Rightarrow \nu_{\ain},\quad\text{ in }M_+((0,\infty]),\label{eq:conv_meas_in}\\
\frac{1}{k_n}\sum_{v}\epsilon_{\Dout_v(n)/b_2(n/k_n)}&\Rightarrow \nu_{\aout},\quad\text{ in }M_+((0,\infty]).\label{eq:conv_meas_out}
\end{align}
This leads to consistency of the Hill estimator \sid{for $\ain$ and $\aout$.}

\section{Preferential Attachment Models.}\label{sec:motiv}
\subsection{Model setup.}\label{subsec:PA}
\sid{Consider $\{G(n), n\geq 1\}$, a growing sequence of preferential
attachment graphs.
The graph $G(n)$ consists of  $n$ nodes, denoted by $ [n]:= \{1,2, \ldots, n\}$,
and $n$ directed edges; the set of edges of $G(n)$ consisting of
ordered pairs of nodes in $[n]$ is denoted by  $E(n)$. 
The initial graph  $G(1)$
consists of one node, labeled node 1, with a self loop.} Thus node 1 has
in- and out-degrees both equal to 1.
For $n\ge 1$,
we obtain a new graph $G(n+1)$ by appending a new node ${n+1}$ \sid{and a
new directed edge} to the existing graph $G(n)$ \sid{according to
probabilistic rules described below.}
For $v\in [n]$,  $(\Din_v(n),\Dout_v(n))$ are the in- and out-degree
of node $v$ in $G(n)$. 
The direction of the new edge in $G(n+1)$ is determined by flipping a
2-sided coin, which has probabilities  
$\alpha \in (0,1)$ and $1-\alpha \equiv \gamma$, such that 
given $G(n)$ and two positive parameters $\deltain, \deltaout>0$ (not necessarily equal):
\begin{itemize}
\item \sid{If the coin comes up heads} with probability $\alpha$, 
\sid{direct the} new edge from the new node ${n+1}$ to the
 existing node $v\in [n]$  with probability depending on the in-degree of $v$ in $G(n)$:
\beqq\label{eq:in}
\PP(v\in [n] \text{ is chosen}) = \frac{\Din_v(n)+\deltain}{(1+\deltain)n}.
\eeqq
\item \sid{If the coin comes up tails} with probability $\gamma$,
\sid{direct the} new edge from an existing node $v\in [n]$ to the new
node ${n+1}$, with probability depending on the out-degree of $v$ in $G(n)$:
\beqq\label{eq:out}
\PP(v\in [n]\text{ is chosen}) = \frac{\Dout_v(n)+\deltaout}{(1+\deltaout)n}.
\eeqq
\end{itemize}
We refer the two scenarios as $\alpha$- and $\gamma$-schemes, respectively.

\subsubsection{Model construction.}\label{subsub:construct}
One way to \sid{formally construct the model which helps with proofs}
 is by \sid{using} independent exponential random variables (r.v.'s).
Define \sid{derived parameters}
\begin{equation}\label{eq:cees}
\cin =\frac{\alpha}{1+\deltain}\qquad\text{and}\qquad \cout = \frac{\gamma}{1+\deltaout},
\end{equation}
and \sid{for $n\geq 1$, we will recursively define what corresponds to the in- and
out-degree sequences as random elements of $(\mathbb{N}_+^2)^\infty$},
\begin{equation}\label{eq:degreeSeq}
\bfD(n) := \bigl((\Din_1(n), \Dout_1(n)),\ldots, (\Din_n(n), \Dout_n(n)), (0,0),\ldots\bigr)
\end{equation}
with initialization
\begin{equation}\label{eq:init}
\bfD(1)=\bigl((1,1),(0,0),\dots \bigr)\end{equation}
\sid{corresponding to assuming $G(0)$ has a single node with a self
  loop.} For $k\ge 1$,  \sid{the recursive definition of $\{\bfD(n)\}$ uses
the variables }
\begin{align}
\be_k^\text{in} &:= ((0,0),\ldots,(0,0),\underbrace{(1,0)}_\text{$k$-th entry}, (0,0),\ldots),\label{eq:ein}\\
\be_k^\text{out} &:=
                   ((0,0),\ldots,(0,0),\underbrace{(0,1)}_\text{$k$-th
                   entry}, (0,0),\ldots),\label{eq:eout} 
\end{align}
and \sid{relies on competitions from exponential alarm clocks based on }
$\{E^{(n)}_k: k\ge 1, n\ge 1\}$,  a sequence of iid standard exponential r.v.'s.
\sid{Assuming $\bfD(n)$ has been given,} $\bfD(n+1)$ requires $\bfD(n)$ and
the $2n$ variables $\{E_j^{(n)}, j=1,\dots,2n\}$
\sid{which are independent of $\bfD(n)$ (which can be checked
  recursively) and we define }
\begin{align*}
\overline{E}^{(n)}_k & :=\frac{E^{(n)}_k}{\frac{\cin}{\cin+\cout}(\Din_k(n)+\deltain)},\qquad k=1, \ldots, n,\\
\overline{E}^{(n)}_k
                     &:=\frac{E^{(n)}_k}{\frac{\cout}{\cin+\cout}(\Dout_k(n)+\deltaout)},\qquad k=n+1, \ldots, 2n.
\end{align*}
Conditionally on $\bfD(n)$, use
the $\{\overline{E}^{(n)}_k: k=1,\ldots,2n\}$ to create
 a competition between exponentially distributed alarm clocks. 
For $\deltain, \deltaout>0$ and $n\ge 1$, define choice variables
\[
L_{n+1} = \sid{\sum_{l=1}^n} l \ind_{\left\{\Ebar^{(n)}_l < \bigwedge_{k=1, k\neq l}^{2n} \Ebar^{(n)}_k,\, 1\le l\le n\right\}}
+ \sid{\sum_{l=n+1}^{2n}} l \ind_{\left\{\Ebar^{(n)}_l < \bigwedge_{k=1, k\neq l}^{2n} \Ebar^{(n)}_k,\, n+1\le l\le 2n\right\}}.
\] 
\sid{So $L_{n+1}$ is the index of the minimum of
$\{\overline{E}^{(n)}_k, 1\leq k\leq 2n\}$ indicating the winner of
the competition.} Also, for $n\ge 1$, define the Bernoulli random variable 
$$
B_{n+1} := \ind_{\left\{\bigwedge_{k=1}^n\Ebar^{(n)}_k
    > \bigwedge_{k=n+1}^{2n}\Ebar^{(n)}_k\right\}}=\ind_{\{L_{n+1}>n\}},
$$
and given $\bfD(n)$, we have
\begin{equation}\label{eq:recursive}
\bfD(n+1)=\bfD(n)  +
(1-B_{n+1})\be_{L_{n+1}}^\text{in} + B_{n+1}\be_{L_{n+1}-n}^\text{out}
+ B_{n+1}\be_{n+1}^\text{in} + (1-B_{n+1})\be_{n+1}^\text{out}.
\end{equation}
\sid{This increments the $L_{n+1}$-st pair by $(1,0)$ if $B_{n+1}=0$ 
and the $(L_{n+1}-n)$-th pair by (0,1) if $B_{n+1}=1$; the first case corresponds to an increase of
in-degree and the second case to an increase of out-degree. The
recursion also assigns to pair $n+1$ either $(1,0)$ or $(0,1)$
depending on the case. This construction expresses $\bfD(n+1)$ as a
function of $\bfD(n)$ and something independent, namely $\{E_j^{(n)},
j=1,\dots,2n\}$ and therefore the process
 $\{\bfD(n), n\geq 1\}$ is an $(\mathbb{N}_+^2)^\infty$-valued Markov
chain. Also, because of the initialization \eqref{eq:init}, a simple
induction argument applied to \eqref{eq:recursive} gives the sum of
the components satisfies}
\beqq\label{eq:sum}
\sum_j \Din_j(n) =\sum_j \Dout_j(n) =n,\quad n\geq 1.\eeqq
Then using \eqref{eq:cees}, \eqref{eq:sum} and standard calculations with exponential
rv's, we have for $v\in [n]$, 
\begin{align}
\PP&\left(\bfD(n+1)=\bfD(n) + \be_v^\text{in} +
                       \be_{n+1}^\text{out}|\bfD(n) )=\PP(L_{n+1} = v
               \middle|\bfD(n)\right) \notag \\
=&\PP \left(\overline E_v^{(n)} <\bigwedge_{k=1,k\neq v}^{2n}\overline E_k^{(n)}\middle|\bfD(n)\right)
=
   \frac{\alpha(\Din_v(n)+\deltain)}{(1+\deltain)n},\label{eq:def_Lin}\\
\intertext{and likewise}
\PP&\left(\bfD(n+1)=\bfD(n) + \be_v^\text{out} +
                \be_{n+1}^\text{in}\middle|\bfD(n)\right)
= \PP(L_{n+1} =n+ v|\bfD(n))\notag \\
=&\PP \left(\overline E_{n+v}^{(n)} <\bigwedge_{k=1,k\neq  n+v}^{2n}\overline
   E_k^{(n)} \middle|\bfD(n)\right)
 = \frac{\gamma(\Dout_v(n)+\deltaout)}{(1+\deltaout)n}.\label{eq:def_Lout}
\end{align}
These probabilities agree with the attachment probabilities
\eqref{eq:in}, \eqref{eq:out} in $\alpha$- and $\gamma$-schemes, respectively.

\subsection{Power-law tails.}
Suppose $G(n)$ is a random graph generated by the dynamics above after $n$ steps. 
Let $N_{i,j}(n)$ be the number of nodes in $G(n)$ with in-degree $i$ and out-degree $j$, i.e.
\beqq\label{eq:defN}
N_{i,j}(n) := \sum_{v\in [n]} \ind_{\left\{\bigl(\Din_v(n), \Dout_v(n)\bigr) = (i,j)\right\}}, 
\eeqq
then $\Nin_i (n) := \sum_{j} N_{i,j}(n)$ and $\Nin_{>i}(n) := \sum_{k>i}\Nin_k(n)$ are the number of nodes in $G(n)$ with in-degree equal to and strictly greater than  $i$, respectively.
A similar definition also applies to out-degrees: $\Nout_j (n) := \sum_{i} N_{i,j}(n)$ and $\Nout_{>j}(n) := \sum_{k>j}\Nout_k(n)$.

It is
shown in \cite[Theorem~3.2]{bollobas:borgs:chayes:riordan:2003} 
using concentration inequalities and martingale methods
that 
for as $n\to\infty$,
\beqq\label{eq:pij}
\frac{N_{i,j}(n)}{n} \convp p_{ij},
\eeqq
where $p_{ij}$ is a probability mass function (pmf) and
\cite{wang:resnick:2016,
  resnick:samorodnitsky:towsley:davis:willis:wan:2016,
  resnick:samorodnitsky:2015} show that 
$p_{ij}$ is jointly regularly varying and so is the associated joint measure.
The analytical form of $p_{ij}$ is given in \cite{bollobas:borgs:chayes:riordan:2003}, but later in Section~\ref{subsec:Nij_conv}, we give another proof
using Section~\ref{sec:embed}'s  embedding technique.

From \cite[Theorem 3.1]{bollobas:borgs:chayes:riordan:2003}, the
\sid{scaled }marginal degree counts $\Nin_i(n)/n$ and $\Nout_j(n)/n$, $i,j\ge 0$, also
converge: 
\begin{align}
\frac{\Nin_0(n)}{n} & \convp \pin_0 = \frac{\alpha}{1+\cin\deltain},\qquad \frac{\Nout_0(n)}{n}  \convp \pout_0 = \frac{\gamma}{1+\cout\deltaout},\label{eq:p0}\\
\frac{\Nin_i(n)}{n} & \convp \pin_i = \frac{\Gamma(i+\deltain)}{\Gamma(i+1+\deltain+\cin^{-1})} \frac{\Gamma(1+\deltain+\cin^{-1})}{\Gamma(1+\deltain)}
\left(\frac{\alpha\deltain}{1+\cin\deltain}+\frac{\gamma}{\cin}\right),\quad i\ge 1,\label{eq:pin}\\
\frac{\Nout_j(n)}{n} & \convp \pout_j = \frac{\Gamma(j+\deltaout)}{\Gamma(j+1+\deltaout+\cout^{-1})} \frac{\Gamma(1+\deltaout+\cout^{-1})}{\Gamma(1+\deltaout)}
\left(\frac{\gamma\deltaout}{1+\cout\deltaout}+\frac{\alpha}{\cout}\right),\quad j\ge 1.\label{eq:pout}
\end{align}
\sid{Both} $\left(\pin_i\right)_{i\ge 0}$ and $\left(\pout_j\right)_{j\ge 0}$ are pmf's and the asymptotic form follows
 from Stirling's formula:
 \begin{align*}
 \pin_i \sim C_{IN} \cdot i^{-(1+\cin^{-1})}, &\qquad  i \to\infty,\\
 \pout_j \sim C_{OUT} \cdot j^{-(1+\cout^{-1})}, &\qquad  j\to\infty.
\end{align*}  

Let $\pin_{>i}=\sum_{k>i}\pin_k$ and $\pout_{>j}=\sum_{k>j}\pout_k$ be the complementary cdf's and
by Scheff\'e's lemma as well as
\cite[Equation (8.4.6)]{vanderHofstad:2017}, we have 
\begin{align}
\frac{\Nin_{>i}(n)}{n}\convp \pin_{>i} &:=
\frac{\Gamma(i+1+\deltain)}{\Gamma(i+1+\deltain+\cin^{-1})}\cin\frac{\Gamma(1+\deltain+\cin^{-1})}{\Gamma(1+\deltain)}
\left(\frac{\alpha\deltain}{1+\cin\deltain}+\frac{\gamma}{\cin}\right),\label{eq:def_pin_over}\\
\frac{\Nout_{>j}(n)}{n}\convp \pout_{>j} &:=
\frac{\Gamma(j+1+\deltaout)}{\Gamma(j+1+\deltaout+\cout^{-1})}\cout\frac{\Gamma(1+\deltaout+\cout^{-1})}{\Gamma(1+\deltaout)}
\left(\frac{\gamma\deltaout}{1+\cout\deltaout}+\frac{\alpha}{\cout}\right),\label{eq:def_pout_over}
\end{align}
so again by Stirling's formula we get from \eqref{eq:def_pin_over} and \eqref{eq:def_pout_over} that 
\begin{align*}
\pin_{>i}  \sim C'_{IN}\cdot  i^{-\cin^{-1}} =: C'_{IN}\cdot  i^{-\ain}, &\qquad  i \to\infty,\\
\pout_{>j} \sim C'_{OUT}\cdot  j^{-\cout^{-1}} =: C'_{OUT}\cdot  j^{-\aout}, &\qquad j\to\infty.
\end{align*}
In other words, the marginal tail distributions of the asymptotic in- and out-degree
sequences in a directed linear preferential attachment model are asymptotic
to power laws with tail indices $\ain\equiv\cin^{-1}$ and $\aout\equiv\cout^{-1}$, respectively. 


\section{Preliminaries: Switched Birth Immigration Processes.}\label{sec:SBI}
In this section, we introduce a pair of switched birth immigration processes (SBI processes). This lays the 
foundation for Section~\ref{sec:embed}, where we embed the in- and out-degree sequences of a fixed network node
into a pair of SBI processes and derive the asymptotic limit of the degree growth.

\subsection{Birth immigration processes.}
We start with a brief review of the birth immigration process.
A linear birth process with immigration (BI process),
$\{Z(t): t\ge 0\}$, having lifetime parameter $\lambda>0$ and
immigration parameter $\theta\ge 0$ is a continuous time Markov
process with state space $\mathbb{N} =\{0,1,2,3,\ldots\}$ and
transition rate  
$$q^Z_{k,k+1} = \lambda k + \theta,\qquad k\ge 0.$$
When $\theta = 0$ there is no immigration and the BI process becomes
a pure birth process and in such cases, the process usually starts from 1.

For $\theta>0$, the BI process starting from 0 can be
constructed from a Poisson process and an independent family of iid
linear birth processes \cite{tavare:1987}.
Suppose that $N_\theta (t)$ is the counting function of homogeneous
Poisson points $0<\tau_1<\tau_2<\ldots$ with rate
$\theta$ and independent of this Poisson process we have
 independent
copies of a linear birth process $\{\zeta_i(t):t\ge 0\}_{i\ge 1}$
with parameter $\lambda>0$ and $\zeta_i(0) = 1$ for $i\ge 1$.  
\sid{The  BI process $Z(t), t\geq 0$ is a shot noise process with $Z(0)=0$
and for $t\ge 0$,}
\beqq\label{eq:defBI}
Z(t) := \sum_{i=1}^\infty
\zeta_i(t-\tau_i)\ind_{\{t\ge\tau_i\}}=\sum_{i=1}^{N_\theta(t) } \zeta_i(t-\tau_i).
\eeqq

Theorem~\ref{thm:tavare} modifies slightly the statement of 
\cite[Theorem~5]{tavare:1987} summarizing  the asymptotic behavior of the BI process.
This is also reviewed in \cite{wang:resnick:2017}.
\begin{Theorem}\label{thm:tavare}
For $\{Z(t):t\ge 0\}$ as in \eqref{eq:defBI}, we have as $t\to\infty$,
\beqq \label{e:sigma}
e^{-\lambda t}Z(t) \convas \sum_{i=1}^\infty W_i e^{-\lambda\tau_i} =:\sigma
\eeqq
where $\{W_i: i\ge 1\}$ are independent unit exponential random
variables satisfying a.s. for each $i\ge 1$,
   $$W_i=\lim_{t\to\infty} e^{-t}\zeta_i(t).$$
The random variable $\sigma$ in \eqref{e:sigma}
is a.s. finite and has a Gamma density given by
\[
f(x) = \frac{1}{\Gamma(\theta/\lambda)}x^{\theta/\lambda-1} e^{-x},\qquad x>0.
\]
\end{Theorem}
\begin{Remark}\label{rmk:tavare}
{\rm
The form of $\sigma$  in \eqref{e:sigma} and its Gamma density is justified in
\cite{tavare:1987, wang:resnick:2017}. 
For a BI process $\{Z'(t)\}_{t\ge 0}$ with $Z'(0) = j\ge 1$, modifying the representation in \eqref{eq:defBI} gives
\[
Z'(t) = \sum_{i=1}^j\zeta_i(t) +\sum_{i=j+1}^\infty
\zeta_i(t-\tau_i)\ind_{\{t\ge\tau_i\}}.
\]
Therefore, $e^{-\lambda t}Z'(t)\convas \sigma'$ where $\sigma'$ has a Gamma density given by
$g(x) = x^{j+\theta/\lambda-1}e^{-x}/\Gamma(j+\theta/\lambda)$, $x>0$.}
\end{Remark}

\subsection{Switched birth immigration processes.}\label{subsec:switcheroo}
\sid{A switched birth immigration (SBI) process uses a Bernoulli
  choice variable to
choose randomly from two independent
BI processes with the same linear transition rates with one starting from $1$
at $t=0$ and the other starting from $0$. A {\it pair\/} of SBI
processes takes two SBI processes which are linked through the same
Bernoulli choice variable.} 

\begin{table}[h]
\begin{tabular}{ |c|cc|cc|}
\hline
Process & $I^{(0)}(t)$ &$ I^{(1)}(t)$ & $O^{(0)}(t)$ & $O^{(1)} (t)$ \\
$t=0 $ & 0  &1 & 1& 0 \\
Rate &\multicolumn{2}{c|}{   $(1-p)(k+\delta_1)$}  &
                                                    \multicolumn{2}{c|}{ $p(k+\delta_2)$}\\
\hline
\end{tabular} \vspace{.1in}
\caption{Ingredients for a pair of switched BI processes.}
\label{tab:sbi}
\end{table}

Suppose that $J$ is a Bernoulli switching random variable with 
\[
\PP(J=1) = p = 1-\PP(J = 0),
\]
and $\{\Ia(t):t\ge 0\}$, $\{\Ib(t):t\ge 0\}$, $\{\Oa(t):t\ge 0\}$, $\{\Ob(t):t\ge 0\}$ are four independent BI processes (also independent of $J$) with
$\Ia(0)=\Ob(0)=0$, $\Ib(0)=\Oa(0)=1$ and transition rates
\begin{align*}
q^{\Ia}_{k,k+1} = (1-p)(k+\delta_1),&\qquad q^{\Ob}_{k,k+1} = p(k+\delta_2),\quad\text{for }k\ge 0,\\
q^{\Ib}_{k,k+1} = (1-p)(k+\delta_1),&\qquad  q^{\Oa}_{k,k+1} = p(k+\delta_2),\quad\text{for }k\ge 1, \,\delta_1,\delta_2>0.
\end{align*}
\sid{See Table \ref{tab:sbi} for quick reminders.}
Then we construct a pair of SBI processes $\{\bigl(I^{(J)}(t),O^{(J)}(t)\bigr):t\ge 0\}$ using five independent ingredients:
\beqq\label{eq:def_sbi}
\bigl(I^{(J)}(t),O^{(J)}(t)\bigr) := (1-J)\bigl(\Ia(t),\Oa(t)\bigr) + J \bigl(\Ib(t),\Ob(t)\bigr),\qquad t\ge 0.
\eeqq

We then consider the convergence of the pair of SBI processes, $\bigl(e^{-(1-p)t}I^{(J)}(t),e^{-pt}O^{(J)}(t)\bigr)$, as $t\to\infty$.
Write a Gamma random variable $X$ with density
$f_X(x) = b^a x^{a-1}e^{-bx}/\Gamma(a)$, $x>0$ and $a,b>0$, as $X\sim \Gamma(a,b)$.
Then from Theorem~\ref{thm:tavare}, Remark~\ref{rmk:tavare} and \eqref{eq:def_sbi}, we have with $X^{(0)}$, $Y^{(0)}$, $X^{(1)}$, $Y^{(1)}$ being
four independent Gamma random variables and $X^{(0)}\sim \Gamma(\delta_0,1)$, $Y^{(0)}\sim \Gamma(1+\delta_1,1)$,
$X^{(1)}\sim \Gamma(1+\delta_0,1)$, $Y^{(1)}\sim \Gamma(\delta_1,1)$, as $t\to\infty$,
\beqq\label{eq:convIO}
\bigl(e^{-(1-p)t}I^{(J)}(t),e^{-pt}O^{(J)}(t)\bigr)\convas (1-J) (X^{(0)}, Y^{(0)}) + J  (X^{(1)}, Y^{(1)}) =:
(X^{(J)}, Y^{(J)}).
\eeqq
Also, $(X^{(J)}, Y^{(J)})$ has joint density
\beqq\label{eq:densIO}
f_{X^{(J)}, Y^{(J)}}(x,y) = (1-p)\,\frac{x^{\delta_0-1}e^{-x}}{\Gamma(\delta_0)}\frac{y^{\delta_1}e^{-y}}{\Gamma(1+\delta_1)}
+ p\,\frac{x^{\delta_0}e^{-x}}{\Gamma(1+\delta_0)}\frac{y^{\delta_1-1}e^{-y}}{\Gamma(\delta_1)},\quad x,y>0.
\eeqq

\section{Embedding Process.}\label{sec:embed}
In order to prove the weak convergence of the sequence of empirical measures in
\eqref{emp_meas}, we need to embed the in- and out-degree sequences $\{\bigl(\Din_v(n), \Dout_v(n)\bigr), v\in [n], n \geq 1\}$
 into a process constructed from pairs of SBI
processes, as specified in Section~\ref{sec:SBI}. The embedding idea is proposed in
 \cite{athreya:ghosh:sethuraman:2008} and has been used in \cite{wang:resnick:2017} to model two different undirected 
 linear preferential attachment models. 

\subsection{Embedding.}\label{subsec:embed}
Here we discuss how to embed the directed network growth model into 
a process constructed from an infinite sequence of SBI pairs.

\subsubsection{Directed network model and SBI processes.}\label{subsubsec:bivarBI}
The building blocks of the embedding procedure is an infinite family of independent BI processes
$$\left\{I_1(t), O_1(t), \Ia_v(t), \Ib_v(t), \Oa_v(t),\Ob_v(t): v\ge 2, t\ge 0\right\},$$
\sid{defined on the same probability space and} satisfying:
\begin{enumerate}
\item[(i)] \sid{$(I_1(0) ,O_1(0))=1$,} $(\Ia_v(0), O_v^{(0)})=(0,1) $
  and $(I_v^{(1)}(0), O_v^{(1)}(0))=(1,0), $ for each $v\ge 2$.
\item[(ii)] Any process labeled with an $I$ is a BI process with transition rates
\[
q^I_{k, k+1} = \frac{\cin}{\cin+\cout}(k+\deltain), \qquad \deltain>0,
\]
and any process labeled with an $O$ is a BI process with transition rates
\[
q^O_{k, k+1} = \frac{\cout}{\cin+\cout}(k+\deltaout)\qquad \deltaout>0.
\]
These hold for $k\ge 0$ when $v\ge 2$ and $k\ge 1$ for $I_1, O_1$.
\end{enumerate}

On $(\mathbb{N}^2)^\infty$, define
$$\sid{\bZ^{(1)}} =\{\bZ^{(1)}_t: t\ge 0\} :=\left\{ \Bigl(\bigl(I_1(t), O_1(t)\bigr),
  (0,0),\ldots\Bigr): t\ge 0\right\}$$  
and the $\sigma$-algebra
$\mathcal{F}^{(1)}_t := \sigma\left\{\bZ^{(1)}_t: 0\le s\le t\right\}$
\sid{so that $\bZ^{(1)}$ is strong Markov with respect to $\{\mathcal{F}^{(1)}_t\}.$  }
Set $T_1 = 0$
and define \sid{the stopping time $T_2$ with respect to
$\{\mathcal{F}^{(1)}_t,t\geq 0\}$ } as
\beqq\label{eq:defT2}
T_2 := \inf\left\{t\ge 0: \text{$\bZ^{(1)}_t$ jumps}\right\}.
\eeqq
Then $T_2$ is the minimum of two independent exponential r.v.'s with means 
$$\left(\frac{\cin}{\cin+\cout}(1+\deltain)\right)^{-1}\qquad\text{and}\qquad \left(\frac{\cout}{\cin+\cout}(1+\deltaout)\right)^{-1}.$$
From \eqref{eq:cees}, we have
\[
\PP[T_2>t]=e^{-(\cin+\cout)^{-1} t}, \qquad t>0.
\]
Let $J_2 := \ind_{\{\text{$O_1$ jumps first}\}} $ so that $\PP[J_2=1]=\gamma$.
Also, let $\widetilde{L}_2$ be index of the 
\sid{$(I,O)$-pair}
that jumps \sid{first} at $T_2$ \sid{
which in this case is $1$. However, note that}
$\left(\widetilde{L}_2, J_2\right)$ determines which one of $I_1$ and $O_1$ will jump at $T_2$,
and $T_2$ is independent of $\left(\widetilde{L}_2, J_2\right)$ by the
property of independent exponential r.v.'s (cf. \cite[Exercise 4.45(a)]{resnick:1992}). 
In addition, we also have $T_2, \widetilde{L}_2, J_2\in
\mathcal{F}^{(1)}_{T_2}$, that is, measurable with respect to $\mathcal{F}^{(1)}_{T_2}$.

Now use \sid{the independent quantities}
$J_2, (\Ia_2, \Oa_2), (\Ib_2,  \Ob_2)$ to define a pair of SBI
processes $(I_2,O_2)= \bigl((I^{(J_2)}_2,O^{(J_2)}_2\bigr)$ 
as in \eqref{eq:def_sbi}. 
Let $\bz_2(t) := \bigl((0,0),
(I_2^{(J_2)}(t), O_2^{(J_2)}(t)), (0,0),\ldots\bigr)$ and  
$$\sid{\bZ^{(2)}}=\{\bZ^{(2)}_t: t\ge 0\} :=\left\{\bZ^{(1)}_{t+T_2} + \bz_2(t): t\ge 0\right\}.$$
Define the $\sigma$-algebra 
$$\mathcal{F}^{(2)}_{t+T_2} := \sigma\left\{\bZ^{(2)}_s: 0\le s\le
  t\right\}\bigvee \mathcal{F}^{(1)}_{T_2},$$ 
\sid{so that $\bZ^{(2)}$ is strong Markov with respect to
  $\{\mathcal{F}^{(2)}_{t+T_2} ,t\geq 0\}$.}
Also, let 
$$\tau_3 := \inf\left\{t\ge 0: \bZ^{(2)}_t\text{ jumps}\right\},\qquad T_3:=T_2+\tau_3,$$
and $J_3 := \ind_{\left\{\text{One of $O_1(T_2+\cdot)$, $O_2^{(J_2)}(\cdot)$ jumps first}\right\}}$. 
Denote the index of the $(I,O)$-pair that jumps at $T_3$ by $\widetilde{L}_3$ and 
write $\PP^{\mathcal{F}^{(1)}_{T_2}}(\cdot):= \PP(\cdot|\mathcal{F}^{(1)}_{T_2})$,
$\PP_{\bz}(\bZ_t\in\cdot):= \PP(\bZ_t\in\cdot|\bZ_0=\bz)$.
Then by the strong Markov property, we have
\[
\PP^{\mathcal{F}^{(1)}_{T_2}}\left(\bZ^{(2)}_t\in\cdot\right)
= \PP_{\bZ^{(1)}_{T_2}+\bz_2(0)}\left(\bZ^{(1)}_t+\bz_2(t)\in\cdot\right).
\]
Therefore, with respect to $\PP^{\mathcal{F}^{(1)}_{T_2}}$, $\tau_3$ is the minimum of 4 independent exponential r.v.'s with means
$\left(\frac{\cin}{\cin+\cout}(I_1(T_2)+\deltain)\right)^{-1}$, $\left(\frac{\cout}{\cin+\cout}(O_1(T_2)+\deltaout)\right)^{-1}$,
$\left(\frac{\cin}{\cin+\cout}(J_2+\deltain)\right)^{-1}$ and $\left(\frac{\cout}{\cin+\cout}(1-J_2+\deltaout)\right)^{-1}$.
Note that $(I_1(T_2), O_1(T_2)) = (2-J_2, 1+J_2)$. 
We then have the following:
\begin{enumerate}
\item $\PP^{\mathcal{F}^{(1)}_{T_2}}(\tau_3>t) = e^{-2(\cin+\cout)^{-1} t}$, $t>0$.
\item $\PP^{\mathcal{F}^{(1)}_{T_2}}(J_3=1)=\gamma$ and $\tau_3$ is independent of $(\widetilde{L}_3, J_3)$ with respect to $\PP^{\mathcal{F}^{(1)}_{T_2}}$.
\item The random variables $T_3, \widetilde{L}_3,J_3 \in
  \mathcal{F}^{(2)}_{T_3} \sid{=\mathcal{F}^{(2)}_{\tau_3 +T_2  }  }$.
\end{enumerate}

\sid{Continue in this way to} use \sid{the conditionally independent quantities } $J_3$, $(\Ia_3, \Oa_3)$ and $( \Ib_3,
\Ob_3)$ to define a pair of SBI processes
$(I_3,O_3)=\bigl(I^{(J_3)}_3,O^{(J_3)}_3\bigr)$ 
as in \eqref{eq:def_sbi}.
In general, for $n\ge 3$, \sid{set}
\begin{align*}
\bZ_t^{(n)} :=& \left(\bigl(I_1(T_n+t), O_1(T_n+t)\bigr), \bigl(I_2^{(J_2)}(T_n-T_2+t), O^{(J_2)}_2(T_n-T_2+t)\bigr),\right.\\
&\left.\quad \ldots, \bigl(I^{(J_n)}_n(t), O^{(J_n)}_n(t)\bigr),(0,0),\ldots\right),\qquad t\ge 0,
\end{align*}
 $\mathcal{F}_{t+T_n}^{(n)}:= \sigma\left\{\bZ_s^{(n)}: 0\le
  s\le t\right\}\sid{
\bigvee } \mathcal{F}_{T_n}^{(n-1)}$,
 $\tau_{n+1}:=\inf\{t\ge 0: \bZ_t^{(n)}\text{ jumps}\}$ and $T_{n+1}:=T_n+\tau_{n+1}$.
Also, define 
\begin{itemize}
\item $J_{n+1}:= \ind_{\left\{\text{One of $O_1(T_n+\cdot)$, $O_k^{(J_k)}(T_n-T_k+\cdot)$, $k=2,\ldots,n$ jumps first}\right\}}$, and
\item $\widetilde{L}_{n+1}$ is the index of the \sid{$(I,O)$-pair}
  that jumps first \sid{among $(I_1(T_n+t),O_1(T_n+t)),
    (I_k(T_n-T_k+t,O_k(T_n-T_k+t),\,k=2,\dots,n$. }
\end{itemize}
Note that with 
$$\bz_n(t) : = \left((0,0),\ldots, \underbrace{\left(I_n^{(J_n)}(t), O_n^{(J_n)}(t)\right)}_{\text{$n$-th pair}}, (0,0),\ldots\right),$$
we have $\bZ_t^{(n)} = \bZ_{\tau_n+t}^{(n-1)} + \bz_n(t)$.
Using the strong Markov property gives
\[
\PP^{\mathcal{F}^{(n-1)}_{T_n}}\left(\bZ^{(n)}_t\in\cdot\right)
= \PP_{\bZ^{(n-1)}_{\tau_n}+\bz_n(0)}\left(\bZ^{(1)}_t+\sum_{k=2}^n\bz_k(t)\in\cdot\right) .
\]
Then with respect to $\mathcal{F}^{(n-1)}_{T_n}$, $\tau_{n+1}$ is the minimum of $2n$ independent exponential r.v.'s with means
\begin{align*}
\left(\frac{\cin}{\cin+\cout}(I_1(T_n)+\deltain)\right)^{-1},& \left(\frac{\cout}{\cin+\cout}(O_1(T_n)+\deltaout)\right)^{-1},\\
\left(\frac{\cin}{\cin+\cout}(I_k^{(J_k)}(T_n-T_k)+\deltain)\right)^{-1},&\left(\frac{\cout}{\cin+\cout}(O_k^{(J_k)}(T_n-T_k)+\deltaout)\right)^{-1},\, k=2,\ldots,n.
\end{align*}
This implies:
\begin{enumerate}
\item The random variable $\tau_{n+1}$ is independent of $(\widetilde{L}_{n+1}, J_{n+1})$ with respect to $\PP^{\mathcal{F}^{(n-1)}_{T_n}}$.
\item The random variables $T_{n+1}, \widetilde{L}_{n+1},J_{n+1} \in \mathcal{F}^{(n)}_{T_{n+1}}$.
\end{enumerate}

Set $\tau_2:= T_2$. Then from this construction follow properties of the distribution of $\{\tau_{n}\}_{n\ge 2}$ and $\{J_n\}_{n\ge 2}$.
\begin{Lemma}\label{lem:Tn}
Suppose $\{T_n\}_{n\ge 1}$, $\{\tau_n\}_{n\ge2}$ and $\{J_n\}_{n\ge 2}$ are defined as above. Then:
\begin{enumerate}
\item[(i)] 
\sid{The sequence $\{J_{n} \}$ is independent of  $\{\tau_{n} \}$.}
\item[(ii)]  \sid{The sequence} $\{J_n\}$ is a sequence of iid Bernoulli random variables with
\beqq\label{eq:J}
\PP(J_n = 1) = \gamma = 1-\PP(J_n=0),\qquad n \ge 2.
\eeqq
\item[(iii)] \sid{The sequence} $\{\tau_n\}_{n\ge 2}$ satisfies
\begin{align}
\{\tau_{n+1}: n\ge 1\} &\stackrel{d}{=} \left\{\frac{E_n}{(\cin+\cout)^{-1}n}, n\ge 1\right\},\label{eq:N}
\end{align}
where $\{E_n: n\ge 1\}$ is a sequence of iid unit exponential random
variables. \sid{So $\{T_n\}$ are the birth times of a linear birth
  process with birth rate $(\cin+\cout)^{-1}$.}
\end{enumerate}
\end{Lemma}
\begin{proof}
For brevity of notation, write
$\lambda^{I_1}_{n} = \frac{\cin}{\cin+\cout}(I_1(T_n)+\deltain)$, $\lambda^{O_1}_{n} = \frac{\cout}{\cin+\cout}(O_1(T_n)+\deltaout)$ and 
for $2\le k\le n, n\ge 2$,
\begin{align*}
\lambda^{I_k}_{n} &= \frac{\cin}{\cin+\cout}(I^{(J_k)}_k(T_n-T_k)+\deltain), \\
\lambda^{O_k}_{n} &= \frac{\cout}{\cin+\cout}(O^{(J_k)}_k(T_n-T_k)+\deltaout).
\end{align*}

At each $T_n$, $n\ge 2$, we start a new
pair of SBI processes $(I_n(\cdot), O_n(\cdot))$ with initial value $(J_n,1-J_n)$ and  
one of $(I_k(\cdot), O_k(\cdot))$, $1\le k\le n-1$ increases by $(1-J_n, J_n)$. 
This corresponds in the network, for instance if $J_n=1$, to one of 
the existing $n-1$ nodes having an out-degree
increase by 1 and a new node $n$ with in-degree 1 and out-degree 0.
Therefore  \sid{(cf. \eqref{eq:sum})},
\begin{align}\label{eq:sum_inout}
I_1(T_n)+\sum_{k=2}^n  I_k^{(J_k)}(T_n-T_k)
 = 
O_1(T_n)+ \sum_{k=2}^n O_k^{(J_k)}(T_n-T_k)  = n.
\end{align} 
Hence,  for $n\ge 2$, \sid{$t_l>0$ and $j_l \in \{0,1\}$ for $l=2,\dots,n+1$,}
\begin{align}\label{eq:JTjoint}
\PP\left(\bigcap_{l=2}^{n+1} [
\tau_{l}>t_{l}, J_{l}=j_{l} ]\right)
&= \EE\left[
\PP^{\mathcal{F}^{(n-1)}_{T_n}}\left(\tau_{n+1}>t_{n+1}, J_{n+1}=j_{n+1}, 
\bigcap_{l=2}^{n} \{\tau_{l}>t_{l}, J_{l}=j_{l}\right)\right]\notag\\
&= \EE\left[\ind_{\bigcap_{l=2}^{n} \{
\tau_{l}>t_{l}, J_{l}=j_{l} \}}
\PP^{\mathcal{F}^{(n-1)}_{T_n}}\left(\tau_{n+1}>t_{n+1}, J_{n+1}=j_{n+1}\right)\right],
\end{align}
since $(\tau_l, J_l,l=2,\dots,n )\in \mathcal{F}^{(n-1)}_{T_n}$. Also, we know that with respect to $\PP^{\mathcal{F}^{(n-1)}_{T_n}}$, $\tau_{n+1}$ is the 
minimum of $2n$ independent exponential r.v.'s and $J_{n+1}$ is independent of $\tau_{n+1}$. Therefore,
\begin{align}\label{eq:cond_tauJ}
\PP^{\mathcal{F}^{(n-1)}_{T_n}}\left(\tau_{n+1}>t_{n+1}, J_{n+1}=j_{n+1}\right)
&=
  \PP^{\mathcal{F}^{(n-1)}_{T_n}}\left(\tau_{n+1}>t_{n+1}\right)\PP^{\mathcal{F}^{(n-1)}_{T_n}}\left(J_{n+1}=j_{n+1}\right). 
\end{align}
Note that 
\begin{align}\label{eq:cond_tau}
 \PP^{\mathcal{F}^{(n-1)}_{T_n}}\left(\tau_{n+1}>t_{n+1}\right) &=  \exp\left\{-t_{n+1} \sum_{k=1}^n \left(\lambda^{I_k}_n + \lambda^{O_k}_n\right)\right\}\notag\\
 &= \exp\left\{-t_{n+1}(\cin+\cout)^{-1}n\right\},
\end{align}
and assuming $j_{n+1}=1$, we have
\begin{align}\label{eq:cond_J}
\PP^{\mathcal{F}^{(n-1)}_{T_n}}\left(J_{n+1}=1\right)
&=  \frac{\sum_{k=1}^n\lambda^{O_k}_n}{\sum_{k=1}^n (\lambda^{I_k}_n+\lambda^{O_k}_n)}= \gamma. 
\end{align}
\sid{So \eqref{eq:JTjoint} becomes (continuing to suppose $j_{n+1}=1$),
$$
\PP\left(\bigcap_{l=2}^{n+1} [
\tau_{l}>t_{l}, J_{l}=j_{l} ]\right)
=\gamma \exp\left\{-t_{n+1}(\cin+\cout)^{-1}n\right\}
\PP\left(\bigcap_{l=2}^{n} [
\tau_{l}>t_{l}, J_{l}=j_{l} ]\right).
$$
If $j_{n+1}=0$, $\gamma $ is replaced by $\alpha$ on the right
side. This is sufficient for the proof of the Lemma.}
\end{proof}

\subsubsection{Embedding.}
The following embedding theorem is similar to those proved in
\cite{athreya:ghosh:sethuraman:2008, wang:resnick:2017} and summarizes how to embed in
the paired SBI process constructions.
\begin{Theorem}\label{thm:embed} 
Suppose that $\{T_n\}_{n\ge 1}$ and $\{\bZ^{(n)}_t: t\ge 0\}$ are as defined in Section~\ref{subsubsec:bivarBI}.
Then in
  $((\mathbb{N}^2)^\infty){}^\infty$,
$$\left\{ \bfD(n), n\geq 1\right\} \stackrel{d}{=} \left\{\bZ^{(n)}_0, n\geq 1\right\}.$$
\end{Theorem}
\begin{proof}
The proof relies on both $\{\bfD(n), n\ge 1\}$ and $\{\bZ^{(n)}_0, n\ge 1\}$ being Markov chains with the same transition probabilities.
It is similar to that of
\cite[Theorem~2.1]{athreya:ghosh:sethuraman:2008} and \cite[Theorem 2]{wang:resnick:2017} which we now outline.

Define
\[
\widetilde{\bd}_j^{(J_n)} := \left((0,0),\ldots, \underbrace{(1-J_n,J_n)}_{\text{$j$-th pair}}, (0,0),\ldots, (0,0),\underbrace{(J_n, 1-J_n)}_{\text{$n$-th pair}}, (0,0),\ldots\right)
\]
Recall that $\widetilde{L}_{n+1}$ is the index of the 
\sid{$(I,O)$-pair} 
that jumps at $T_{n+1}$. Then we have
\begin{align}\label{eq:bfDn}
\bZ^{(n+1)}_0=&\bZ^{(n)}_0+\widetilde{\bd}_{\widetilde{L}_{n+1}}^{(J_{n+1})}.
\end{align}
This expresses  $\bZ^{(n+1)}_0$ as a function of
  $\mathcal{F}^{(n-1)}_{T_n}$-measurable random elements and random
  elements independent of $\mathcal{F}^{(n-1)}_{T_n}$, namely:
\begin{enumerate}
\item $\bZ^{(n)}_0 \in
\mathcal{F}^{(n-1)}_{T_n}$;
\item $J_{n+1}$ which is independent of $\mathcal{F}^{(n-1)}_{T_n}$
(by Lemma~\ref{lem:Tn}; see \eqref{eq:cond_J});
\item  $\widetilde{L}_{n+1}$ which is a
function of $(\lambda^{I_k}_n+\lambda^{O_k}_n, k=2,\dots,n) \in
\mathcal{F}^{(n-1)}_{T_n} $ and conditionally on
$\mathcal{F}^{(n-1)}_{T_n}$, $2n$ i.i.d exponential r.v.s which are
independent of $\mathcal{F}^{(n-1)}_{T_n} $.
\end{enumerate}
Hence, both $\{ \bfD(n),n\geq 1\}$ and $\{\bZ^{(n)}_0, n \geq 1\}$
are Markov on the state space $(\mathbb{N}^2)^\infty$.

When $n =1$, 
\begin{align*}
\bZ^{(1)}_0 &= \Bigl(\big(I_1(0), O_1(0)\bigr),(0,0),\ldots\Bigr) = \bigl((1,1), (0,0),\ldots\bigr) \\
&= \Bigl(\bigl(\Din_1(1),\Dout_1(1)\bigr),(0,0),\ldots\Bigr) = \bfD(1),
\end{align*}
so to prove equality in distribution for any $n$,
it suffices to verify that the transition probability from
$\bZ^{(n)}_0$ to $\bZ^{(n+1)}_0$ is the same as that
from 
${\bfD}(n)$ to ${\bfD}(n+1)$ \sid{which is given in \eqref{eq:def_Lin} and 
\eqref{eq:def_Lout}.} 
In the SBI setup, applying Lemma~\ref{lem:Tn} gives
for any $2\le v \le n$,
\begin{align*}
{\PP}^{\mathcal{F}^{(n-1)}_{T_n} }\Big(\bZ^{(n+1)}_0 =&\bZ^{(n)}_0
  +\be^\text{in}_v +\be^\text{out}_{n+1})\Big) 
={\PP}^{\mathcal{F}^{(n-1)}_{T_n} }\Bigl(J_{n+1}=0,\widetilde{L}_{n+1} =v\Bigr)\\
=& \frac{\frac{\cin}{\cin+\cout}(I^{(J_v)}_v(T_n-T_v) + \deltain)}{(\cin+\cout)^{-1}n} 
= \alpha\frac{I^{(J_v)}_v(T_n-T_v)     +\deltain}{(1+\deltain)n},\\
{\PP}^{\mathcal{F}^{(n-1)}_{T_n} }\Big(\bZ^{(n+1)}_0 =&\bZ^{(n)}_0
  +\be^\text{in}_{n+1} +\be^\text{out}_{v})\Big)
={\PP}^{\mathcal{F}^{(n-1)}_{T_n} }\Bigl(J_{n+1}=1,\widetilde{L}_{n+1} =v\Bigr)\\
=& \frac{\frac{\cout}{\cin+\cout}(O^{(J_v)}_v(T_n-T_v) + \deltaout)}{(\cin+\cout)^{-1}n} 
= \gamma\frac{ O^{(J_v)}_v(T_n-T_v) +\deltaout}{(1+\deltaout)n}.\end{align*}
For $2\leq v \leq n$, this agrees with the transition probabilities in \eqref{eq:def_Lin}
and
 \eqref{eq:def_Lout} respectively; the case for $v=1$ is similar.
\end{proof}

\subsection{Asymptotic properties.}\label{subsec:asy}
With the embedding technique specified in Section~\ref{subsec:embed},
the asymptotic behavior of the in- and out-degree growth in a
preferential attachment model can be characterized explicitly. 
These asymptotic properties then help us derive weak convergence of
the empirical measure. 
For brevity of notation, we will write $I^{(J_v)}_v$, $O^{(J_v)}_v$ as $I_v$, $O_v$, $v\ge 2$, in the rest of this paper.

\subsubsection{Convergence of the in- and out-degrees for a fixed node.}
We first consider the asymptotic behavior of the in- and out-degrees for a fixed node, i.e. $(\Din_v(n), \Dout_v(n))$ for a fixed $v$.
To do this, we make use of the embedding results in Theorem~\ref{thm:embed}, which translates the convergence of the degrees to the setting of
$\bigl\{\bigl(I_v(t-T_v),O_v(t-T_v)\bigr): t\ge T_v\bigr\}_{1\le v\le n}$. Results are summarized in Theorem~\ref{thm:fixed_pair}.
\begin{Theorem}\label{thm:fixed_pair}
Suppose that $\{T_n:n\ge 1\}$ and $\{J_n: n\ge 2\}$ are as defined in Section~\ref{subsubsec:bivarBI}. Then:
\begin{enumerate}
\item[(i)] The birth times $\{T_n\}_{n\ge 1}$ satisfy that as $n\to\infty$,
\begin{align}\label{eq:conv_Tn}
n\cdot e^{-(\cin+\cout)^{-1}T_n}\convas W \qquad\text{and}\qquad W\sim \text{Exp}(1).
\end{align}
\item[(ii)] Let $(\sigma^{\text{in}}_1, \sigma^{\text{out}}_1)$ be a pair of independent Gamma random variables with densities
\[
f_{\sigma^{\text{in}}_1}(x) = \frac{x^{\deltain} e^{-x}}{\Gamma(1+\deltain)}\quad \text{and}\quad 
f_{\sigma^{\text{out}}_1}(x) = \frac{x^{\deltaout} e^{-x}}{\Gamma(1+\deltaout)},\, x>0, \text{ respectively,}
\]
and for each $v\ge 2$,
$\bigl(\sigma^{\text{in}}_v,\sigma^{\text{out}}_v\bigr)$ have joint density
\begin{align}\label{eq:joint_dens}
f_{\bigl(\sigma^{\text{in}}_v,\sigma^{\text{out}}_v\bigr)}(x, y) 
&= \alpha \frac{x^{\deltain-1} e^{-x}}{\Gamma(\deltain)} \frac{y^{\deltaout} e^{-y}}{\Gamma(1+\deltaout)}
+\gamma\frac{x^{\deltain} e^{-x}}{\Gamma(1+\deltain)} \frac{y^{\deltaout-1} e^{-y}}{\Gamma(\deltaout)},\quad x,y>0.
\end{align}
Then for a fixed $v\ge 1$, we have, with $W$ defined as in \eqref{eq:conv_Tn},
\begin{align}
\left(\frac{\Din_v(n)}{n^{\cin}},\, \frac{\Dout_v(n)}{n^{\cout}}\right)\, &\convw \, 
\left(\frac{\sigma^{\text{in}}_v e^{-\frac{\cin}{\cin+\cout} T_v}}{W^{\cin}},\, \frac{\sigma^{\text{out}}_v e^{-\frac{\cout}{\cin+\cout}  T_v}}{W^{\cout}}\right)\quad n\to\infty.\label{eq:conv_pair}
\intertext{Also, setting $\Din_v(n) = 0 = \Dout_v(n)$ for all $v\ge n+1$, we get as $n\to\infty$,}
\left(\max_{v\ge 1}\frac{\Din_v(n)}{n^{\cin}},\, \max_{v\ge 1}\frac{\Dout_v(n)}{n^{\cout}}\right)\, &\convw \, 
\left(\max_{v\ge 1}\frac{\sigma^{\text{in}}_v e^{-\frac{\cin}{\cin+\cout}  T_v}}{W^{\cin}},\, \max_{v\ge 1}\frac{\sigma^{\text{out}}_v 
e^{-\frac{\cout}{\cin+\cout} T_v}}{W^{\cout}}\right).
\label{eq:conv_pair_max}
\end{align}
Here $T_v$, 
$(\sigma^{\text{in}}_v, \sigma^{\text{out}}_v)$ \sid{and $W$ are independent} for all $v\ge 2$.
\end{enumerate}
\end{Theorem}
\tw{
\begin{Remark}
{\rm According to the embedding results in Theorem~\ref{thm:embed}, \eqref{eq:conv_pair} 
also implies that there exists random variables $D^{(1)}_v$, $D^{(2)}_v$, $v\ge1$, on the space of $(\Din_v(n), \Dout_v(n))_{v\ge 1}$ 
satisfying $D^{(1)}_v \stackrel{d}{=} W^{-\cin}\sigma^{\text{in}}_v e^{-\frac{\cin}{\cin+\cout} T_v}$
and $D^{(2)}_v \stackrel{d}{=} W^{-\cout}\sigma^{\text{out}}_v e^{-\frac{\cout}{\cin+\cout} T_v}$, $v\ge 1$,
such that as $n\to\infty$,
\[
\left(\frac{\Din_v(n)}{n^{\cin}},\, \frac{\Dout_v(n)}{n^{\cout}}\right)\, \convas \, \left(D^{(1)}_v,D^{(2)}_v\right).
\]
}
\end{Remark}}
\begin{proof}
(i) From Lemma~\ref{lem:Tn}(i),  $\{T_n:n\ge 1\}$ \sid{are jump times
  of  a pure birth process} starting from $1$ and transition rate 
\[ 
q_{j,j+1} = (\cin+\cout)^{-1} j, \qquad j\ge 1.
\]
Therefore, \eqref{eq:conv_Tn} follows from applying the known convergence results of linear birth processes; see
\cite[Theorem~5.11.4]{resnick:1992} and
\cite{kendall:1966, waugh:1971}, among other sources.

\medskip

\noindent (ii) By Theorem~\ref{thm:embed}, \sid{to show}
\eqref{eq:conv_pair}, it suffices to show that as $n\to\infty$,
\begin{align}
\left(\frac{I_v(T_n-T_v)}{n^{\cin}},\, \frac{O_v(T_n-T_v)}{n^{\cout}}\right)\, &\convas \, 
\left(\frac{\sigma^{\text{in}}_v e^{-\frac{\cin}{\cin+\cout}
                                                                                 T_v}}{W^{\cin}},\, \frac{\sigma^{\text{out}}_v e^{-\frac{\cout}{\cin+\cout} T_v}}{W^{\cout}}\right),\label{eq:conv_pair_birth}. 
\end{align}
With \eqref{eq:conv_Tn} available, \sid{we} prove
\eqref{eq:conv_pair_birth} by showing the convergence of
\[
\left(e^{-\frac{\cin}{\cin+\cout}(t-T_v)}I_v(t-T_v), e^{-\frac{\cout}{\cin+\cout}(t-T_v)}O_v(t-T_v)\right),
\]
as $t\to\infty$.
According to the construction of the processes $\{\bigl(I_v(t-T_v), O_v(t-T_v): t\ge T_v\bigr)\}_{v\ge 1}$, 
we know that $(I_1(0), O_1(0))= (1,1)$.
Then applying the convergence result of a BI process in Remark~\ref{rmk:tavare}, we have for independent
 $(\sigma^{\text{in}}_1, \sigma^{\text{out}}_1)\sim \bigl(\Gamma(1+\deltain,1),\,\Gamma(1+\deltaout,1) \bigr)$,
 \[
\left(e^{-\frac{\cin}{\cin+\cout}t}I_1(t), e^{-\frac{\cout}{\cin+\cout}t}O_v(t)\right)\convas (\sigma^{\text{in}}_1, \sigma^{\text{out}}_1),\qquad t\to\infty.
\]
Moreover, it follows from \eqref{eq:convIO} and \eqref{eq:densIO} that
\beqq\label{eq:conv_IOt}
\left(e^{-\frac{\cin}{\cin+\cout}(t-T_v)}I_v(t-T_v), e^{-\frac{\cout}{\cin+\cout}(t-T_v)}O_v(t-T_v)\right)\convas (\sigma^{\text{in}}_v, \sigma^{\text{out}}_v),\qquad t\to\infty,
\eeqq
with  
$\sigma^{\text{in}}_v$ and $\sigma^{\text{out}}_v$ having the joint density as in \eqref{eq:joint_dens}.

Replacing $t$ with $T_n$ in \eqref{eq:conv_IOt} gives
\beqq\label{eq:conv_IOTn}
\left(\frac{I_v(T_n-T_v)}{e^{\frac{\cin}{\cin+\cout} T_n} },\, \frac{O_v(T_n-T_v)}{e^{\frac{\cout}{\cin+\cout} T_n} }\right)\convas \left(\sigma^{\text{in}}_v e^{-\frac{\cin}{\cin+\cout} T_v}, \sigma^{\text{out}}_v e^{-\frac{\cout}{\cin+\cout} T_v}\right), \quad \text{as }n\to\infty.
\eeqq
Therefore, combining \eqref{eq:conv_Tn} and \eqref{eq:conv_IOTn} gives \eqref{eq:conv_pair}.
For $v\ge 2$, the independence of $(\sigma^{\text{in}}_v,
\sigma^{\text{out}}_v)$ and $T_v$ follows from the construction and
\sid{the independence from $W$ follows from \cite[p. 443]{resnick:1992}};
this completes the proof of \eqref{eq:conv_pair_birth}.

\medskip 

\noindent (iii) We verify \eqref{eq:conv_pair_max}
 by showing that as $n\to\infty$,
\beqq\label{eq:conv_pair_max_IOTn}
\left(\max_{v\ge 1}\frac{I_v(T_n-T_v)}{e^{\frac{\cin}{\cin+\cout} T_n} },\, \max_{v\ge 1}\frac{O_v(T_n-T_v)}{e^{\frac{\cout}{\cin+\cout} T_n} }\right)\convas \left(\max_{v\ge 1}\sigma^{\text{in}}_v e^{-\frac{\cin}{\cin+\cout} T_v}, \max_{v\ge 1}\sigma^{\text{out}}_v e^{-\frac{\cout}{\cin+\cout} T_v}\right).
\eeqq
Then combining \eqref{eq:conv_pair_max_IOTn} with \eqref{eq:conv_Tn}
gives \sid{the result.}
We use the proof machinery in \cite[Proposition~3.1]{athreya:ghosh:sethuraman:2008} to show \eqref{eq:conv_pair_max_IOTn}, which
is summarized in the following lemma.
\begin{Lemma}\label{lemma:max}
Let ${a_{n,i} : 1 \le i \le n}_{n\ge 1}$ be a double array of non-negative numbers
such that
\begin{enumerate}
\item[(1)] For all $i\ge 1$, $\lim_{n\to\infty} a_{n,i} = a_i<\infty$,
\item[(2)] $\sup_{n\ge 1}a_{n,i}\le b_i<\infty$ and
\item[(3)] $\lim_{i\to\infty}b_i = 0$.
\end{enumerate} 
Then $\max_{1\le i\le n} a_{n,i}\to \max_{i\ge 1} a_i$, as $n\to\infty$.
\end{Lemma}
First note that for each $v\ge 1$,
\begin{align*}
I_v(T_n-T_v)e^{-\frac{\cin}{\cin+\cout} (T_n-T_v)} &\le \sup_{t\ge 0} I_v(t)e^{-\frac{\cin}{\cin+\cout} t} =: \widetilde{I}_v,\\
O_v(T_n-T_v)e^{-\frac{\cout}{\cin+\cout} (T_n-T_v)} &\le \sup_{t\ge 0} O_v(t)e^{-\frac{\cout}{\cin+\cout} t} =: \widetilde{O}_v.
\end{align*}
Let $a^I_{n,v} := I_v(T_n-T_v) e^{-\frac{\cin}{\cin+\cout} T_n}$, $a^O_{n,v} := O_v(T_n-T_v) e^{-\frac{\cout}{\cin+\cout} T_n}$ for $1\le v\le n$, and
 $b^I_v := \widetilde{I}_ve^{-\frac{\cin}{\cin+\cout} T_v}$, $b^O_v := \widetilde{O}_ve^{-\frac{\cout}{\cin+\cout} T_v}$
for $v\ge 1$. Then Lemma~\ref{lemma:max}(1) is satisfied by \eqref{eq:conv_IOTn}.
Also, for each $v\ge 1$, $\sup_{n\ge 1}a^I_{n,v} \le b^I_v$ and $\sup_{n\ge 1}a^O_{n,v} \le b^O_v$, which satisfies the criterion in Lemma~\ref{lemma:max}(2).

Following the proof of \cite[Theorem~1.1]{athreya:ghosh:sethuraman:2008}, we check the condition in Lemma~\ref{lemma:max}(3) by proving the claim that almost surely, for all $\epsilon>0$,
\beqq\label{eq:claim}
\widetilde{I}_v\le \epsilon v^{c_\text{in}},\quad \text{and}\quad \widetilde{O}_v\le \epsilon v^{c_\text{out}}, \quad\text{for all large $v$.}
\eeqq
Then as $\epsilon$ is arbitrary, it follows from \eqref{eq:conv_Tn} that 
$b^I_v\to 0$ and $b^O_v\to 0$ a.s. as $v\to\infty$. This completes
checking the three criteria in Lemma~\ref{lemma:max} and therefore
leads to \eqref{eq:conv_pair_max}.

To show \eqref{eq:claim}, we use Markov's inequality: for any $r,r'>0$ and $v\ge 2$,
\begin{align*}
\PP(\widetilde{I}_v\ge \epsilon v^{c_\text{in}})&\le \EE(\widetilde{I}_2^r)/(\epsilon^r v^{rc_\text{in}}),\\
\PP(\widetilde{O}_v\ge \epsilon v^{c_\text{out}})&\le \EE(\widetilde{O}_2^{r'})/(\epsilon^{r'}v^{r'c_\text{out}}),
\end{align*}
since $I_v$, $O_v$, $v\ge 2$ are iid SBI processes.
Hence, if we have
\beqq\label{eq:claim2}
\EE(\widetilde{I}_2^r)<\infty\,\text{ and }\,\EE(\widetilde{O}_2^r)<\infty,\qquad\text{for }\, r>c_\text{in}^{-1}, r'>c_\text{out}^{-1},
\, \text{respectively},
\eeqq
then by Borel-Cantelli, the claim in \eqref{eq:claim} is justified.
To prove \eqref{eq:claim2},
let
\begin{align*}
\widetilde{\Ia}_2:= \sup_{t\ge 0} \Ia_2(t)e^{-\frac{\cin}{\cin+\cout} t}, &\qquad \widetilde{\Ib}_2:= \sup_{t\ge 0} \Ib_2(t)e^{-\frac{\cin}{\cin+\cout} t},\\
\widetilde{\Oa}_2:= \sup_{t\ge 0} \Oa_2(t)e^{-\frac{\cout}{\cin+\cout} t}, &\qquad \widetilde{\Ob}_2:= \sup_{t\ge 0} \Ob_2(t)e^{-\frac{\cout}{\cin+\cout} t},
\end{align*}
then by the construction of $\bigl(I_2(\cdot),O_2(\cdot)\bigr)$, we have
\begin{align*}
\EE(\widetilde{I}_2^r) &= \alpha\EE(\widetilde{\Ia}_2^r)+\gamma\EE(\widetilde{\Ib}_2^r)<\infty,\\
\EE(\widetilde{I}_2^{r'}) &= \alpha\EE(\widetilde{\Oa}_2^{r'})+\gamma\EE(\widetilde{\Ob}_2^{r'})<\infty,
\end{align*}
using the assumption that $\Ia_2$, $\Ib_2$, $\Oa_2$ and $\Ob_2$ are independent BI processes so that
results in \cite[Proposition 2.6]{athreya:ghosh:sethuraman:2008} are still applicable here. 
This completes the proof of \eqref{eq:conv_pair_max_IOTn}.
\end{proof}


\section{Convergence Results on Joint Degree Distributions.}\label{sec:conv}
\subsection{Convergence of the joint degree counts.}\label{subsec:Nij_conv}
Now we analyze the convergence of the joint empirical distribution of the in- and out-degrees
$\{\bigl(\Din_v(n), \Dout_v(n)\bigr): v\in [n]\}$, using the SBI embedding technique.
Let $B(a,p)$ be a negative binomial integer valued random variable with parameters $a>0$ and $p\in(0,1)$ (abbreviated as $NB(a,p)$), and the generating function 
of $B(a,p)$ is
\[
\EE\left(s^{B(a,p)}\right) = p^a (1-(1-p)s)^{-a},\qquad 0\le s\le 1.
\]
\tw{We also use the notation $B(a, Z)$ to represent a r.v. having a mixture distribution such that the second parameter of the negative binomial
r.v. is randomized by an independent r.v. $Z$.}
\begin{Theorem}\label{thm:conv_Nij}
Let $N_{i,j}(n)$ be the number of nodes with in-degree $i$ and out-degree $j$ in graph $G(n)$, then we have
\beqq\label{eq:conv_Nij}
\frac{N_{i,j}(n)}{n} \convp \PP\bigl((\calI,\calO)=(i,j)\bigr),\qquad \text{as }n\to\infty.
\eeqq
The limit pair $(\calI,\calO)$ can be represented in distribution as:
\beqq\label{eq:conv_pij}
(\calI,\calO) \stackrel{d}{=} (1-J)(X_1, 1+Y_1) + J(1+X_2, Y_2),
\eeqq
where 
\begin{enumerate}
\item[(i)] $J$ is a Bernoulli switching variable with $\PP(J=1)=1-\PP(J=0)=\gamma$.
\item[(ii)] Suppose $\{B^{(1)}(\delta_1,p): p\in(0,1)\}$, $\{B^{(2)}({\delta'_1}, p): p\in(0,1)\}$, $\{\widetilde{B}^{(1)}(\delta_2,p): p\in(0,1)\}$ and $\{\widetilde{B}^{(2)}(\delta'_2,p): p\in(0,1)\}$, $\delta_1,\delta'_1,\delta_2,\delta'_2>0$, are four independent families of negative binomial
variables, then 
\begin{subequations}\label{eq:negbin}
\beqq\label{eq:negbin1}
(X_1, Y_1) = \left(B^{(1)}\left(\deltain,e^{-\cin T}\right), \widetilde{B}^{(1)}\left(1+\deltaout,e^{-\cout T}\right)\right),
\eeqq
\beqq\label{eq:negbin2}
(X_2, Y_2) = \left(B^{(2)}\left(1+\deltain,e^{-\cin T}\right), \widetilde{B}^{(2)}\left({\deltaout}, e^{-\cout T}\right)\right),
\eeqq
\end{subequations}
with $T$ being an exponential random variable with unit mean, independent of $J$, $B^{(1)}$, $B^{(2)}$, $\widetilde{B}^{(1)}$ and $\widetilde{B}^{(2)}$.
\end{enumerate}
\end{Theorem}
\begin{Remark}
{\rm
Theorem~\ref{thm:conv_Nij} coincides with the known results proven in \cite{resnick:samorodnitsky:2015,resnick:samorodnitsky:towsley:davis:willis:wan:2016},
since $e^{\cin T}$ is a Pareto random variable on $[1,\infty)$ with index $\cin^{-1}$, denoted by $Z$, and 
$e^{\cout T} = Z^{a}$, with $a:=\cout/\cin$.
}
\end{Remark}
\begin{proof}
The proof of \cite[Lemma 3.1]{wang:resnick:2015} verifies that
\[
\left|\frac{N_{i,j}(n)}{n}-\frac{\EE(N_{i,j}(n))}{n}\right| \convp 0, \qquad \text{as }n\to\infty.
\]
Hence, we are left to examine the difference $|\EE(N_{i,j}(n))/n-\PP\bigl((\calI,\calO)=(i,j)\bigr)|$. By the embedding results in Theorem~\ref{thm:embed}, we have
\begin{align}
\frac{\EE(N_{i,j}(n))}{n} &= \EE\left\lbrace \frac{1}{n}\sum_{v\in [n]} \ind_{\bigl\{\bigl(\Din_v(n),\Dout_v(n)\bigr)=(i,j)\bigr\}} \right\rbrace
= \frac{1}{n}\sum_{v\in [n]}
  \PP\Bigl(\bigl(\Din_v(n),\Dout_v(n)\bigr)=(i,j)\Bigr)\nonumber \\
&\sid{=\frac{1}{n}\sum_{v=1}^n
       \PP\left[\bigl(I_v(T_n-T_v),O_v(T_n-T_v)\bigr)=(i,j)\right] }.
\label{eq:nij_step1}
\end{align}

Suppose that $\{B^{(1)}_v(\deltain,p): v\ge 1\}$, $\{B^{(2)}_v(1+\deltain,p): v\ge 1\}$,
$\{\widetilde{B}^{(1)}_v(1+\deltaout, p): v\ge 1\}$ and $\{\widetilde{B}^{(2)}_v(\deltaout, p): v\ge 1\}$
are four independent sequences of negative binomial r.v.'s with given parameters.
Then by the distribution of a BI process (cf. \cite[Equation (2.2)]{tavare:1987} and \cite[Theorem 3.11]{ford:2009}), we have
for any $v\ge 2$, $t\ge 0$ and $k\ge 0$,
\begin{subequations}\label{eq:nij_step2}
\beqq
\PP(\Ia_v(t) = k) = \PP\left[B^{(1)}_v\left(\deltain, e^{-\frac{\cin}{\cin+\cout} t}\right)=k\right], \label{eq:nij_step2a1}
\eeqq
\beqq
\PP(\Ib_v(t) = k) = \PP\left[1+ B^{(2)}_v\left(1+\deltain, e^{-\frac{\cin}{\cin+\cout} t}\right)=k\right], \label{eq:nij_step2a2}
\eeqq
\beqq
\PP(\Oa_v(t)=k) =\PP\left[1+ \widetilde{B}^{(1)}_v\left(1+\deltaout, e^{-\frac{\cout}{\cin+\cout} t}\right)=k\right], \label{eq:nij_step2b1}
\eeqq
\beqq
\PP(\Ob_v(t)=k) = \PP\left[\widetilde{B}^{(2)}_v\left(\deltaout,
    e^{-\frac{\cout}{\cin+\cout}
      t}\right)=k\right],\label{eq:nij_step2b2} 
\eeqq
\end{subequations}
\sid{and note the quantities on the right do not depend on $v$.}
Also, recall that $\bigl(I_v(t),O_v(t)\bigr)_{v\ge 2}$, $t\ge 0$, are \tw{identically distributed} such that,
\begin{align*}
I_v(t) &= (1-J_v) \Ia_v(t) + J_v\Ib_v(t),\quad
O_v(t) = (1-J_v)\Oa_v(t) + J_v \Ob_v(t) .
\end{align*}
\tw{Since for $v\ge 2$, the processes $\Ia_v$, $\Ib_v$, $\Oa_v$ and $\Ob_v$ are independent from each other,}
we then define for any $v\ge 2$,
\begin{align*}
\mathcal{B}_v^{(n)}:=\bigl(&(1-J_v)B^{(1)}_v(\deltain, e^{-(T_n-T_v)})+J_v(1+ B^{(2)}_v(1+\deltain, e^{-(T_n-T_v)}), \\
& (1-J_v)(1+ \widetilde{B}^{(1)}_v(1+\deltaout, e^{-(T_n-T_v)})+J_v(\widetilde{B}^{(2)}_v(\deltaout, e^{-(T_n-T_v)} \bigr),
\end{align*}
\sid{and}  \eqref{eq:nij_step1} becomes,
\begin{align}
& \frac{1}{n}\EE(N_{i,j}(n))
 = \frac{1}{n}\sum_{v=1}^n \PP\left[\bigl(I_v(T_n-T_v),O_v(T_n-T_v)\bigr)=(i,j)\right]\notag \\
 = & \frac{1}{n}\sum_{v=1}^n \PP\left[\mathcal{B}^{(n)}_v= (i,j)\right]
+ \frac{1}{n} \left(\PP\left[\bigl(I_1(T_n),O_1(T_n)\bigr)=(i,j)\right]- \PP\left[\mathcal{B}^{(n)}_1=(i,j)\right]\right). \label{eq:nij_step3}
\end{align}
The last step \sid{is necessitated by}  the construction \sid{since} $(I_1(t), O_1(t))$ is a pair of independent BI processes, which is
different from the rest of the $(I_v(\cdot), O_v(\cdot))_{v\ge 2}$ pairs. Here this difference is inconsequential because
as $n\to\infty$,
\[
\frac{1}{n} \Bigl|\PP\left[\bigl(I_1(T_n),O_1(T_n)\bigr)=(i,j)\right]-
\PP\left[\mathcal{B}_1^{(n)}=(i,j)\right]\Bigr| 
\le \frac{2}{n}\to 0.
\]
So we only need to consider the first term in
\eqref{eq:nij_step3}. Let $U_n$ be a random variable uniformly
distributed on $[n-1]$ and independent of the rest. Then
\begin{align*}
\frac{1}{n}&\sum_{v=1}^n \PP\left[\mathcal{B}^{(n)}_v= (i,j)\right] \\
=& \alpha \frac{1}{n}\sum_{v=1}^n\PP\left[\bigl(B^{(1)}_v(\deltain, e^{-\frac{\cin}{\cin+\cout}(T_n-T_v)}), 
1+ \widetilde{B}^{(1)}_v(1+\deltaout, e^{-\frac{\cout}{\cin+\cout}(T_n-T_v)})\bigr)=(i,j)\right]\\
&+\gamma \frac{1}{n}\sum_{v=1}^n\PP\left[\bigl(1+ B^{(2)}_v(1+\deltain, e^{-\frac{\cin}{\cin+\cout}(T_n-T_v)}), 
\widetilde{B}^{(2)}_v(\deltaout, e^{-\frac{\cout}{\cin+\cout}(T_n-T_v)} \bigr)=(i,j)\right]\\
=& \alpha \left(1-\frac{1}{n}\right)\PP\left[\bigl(B^{(1)}_1(\deltain, e^{-\frac{\cin}{\cin+\cout}(T_n-T_{U_n})}), 
1+ \widetilde{B}^{(1)}_1(1+\deltaout, e^{-\frac{\cout}{\cin+\cout}(T_n-T_{U_n})})\bigr)=(i,j)\right]\\
&+ \gamma \left(1-\frac{1}{n}\right)\PP\left[\bigl(1+ B^{(2)}_1(1+\deltain, e^{-\frac{\cin}{\cin+\cout}(T_n-T_{U_n})}), 
\widetilde{B}^{(2)}_1(\deltaout, e^{-\frac{\cout}{\cin+\cout}(T_n-T_{U_n})} \bigr)=(i,j)\right]\\
&+ \frac{1}{n}\PP\left[\mathcal{B}^{(n)}_n= (i,j)\right],
\end{align*} 
\sid{since the distributions of
$B_v^{(1)},\widetilde{B}_v^{(1)},B_v^{(2)},\widetilde{B}_v^{(1)}$ do
not depend on $v$.}
Let $T$ be a unit exponential random variable that is independent of
$I_v, O_v$, $v\ge 1$. 
\sid{A variant of the Renyi representation for exponential order
  statistics} (see  \cite[Theorem 
3.14]{ford:2009} for details) gives  
\beqq\label{eq:nij_step4}
T_n-T_{U_n} \stackrel{d}{=} \frac{T}{(\cin+\cout)^{-1}}.
\eeqq
Define a Bernoulli random variable $J$ that is independent from $T$, $B^{(1)}_1$, $B^{(2)}_1$,
$\widetilde{B}^{(1)}_1$ and $\widetilde{B}^{(2)}_1$ with 
$\PP(J=1)=\gamma=1-\PP(J=0)$.
\tw{Then applying \eqref{eq:nij_step4} therefore gives
\begin{align*}
\frac{1}{n}&\sum_{v=1}^n \PP\left[\mathcal{B}^{(n)}_v= (i,j)\right]\\
=& \alpha\left(1-\frac{1}{n}\right)\PP\left[\bigl(B^{(1)}_1(\deltain, e^{-\cin T}),
1+ \widetilde{B}^{(1)}_1(1+\deltaout, e^{-\cout T})\bigr)=(i,j)\right]\\
&+ \gamma\left(1-\frac{1}{n}\right) \PP\left[\bigl(1+ B^{(2)}_1(1+\deltain, e^{-\cin T}), 
\widetilde{B}^{(2)}_1(\deltaout, e^{-\cout T} \bigr)=(i,j)\right] + \frac{1}{n}\left[\mathcal{B}^{(n)}_n= (i,j)\right]\\
=& \left(1-\frac{1}{n}\right)\PP\left[(1-J)\bigl(B^{(1)}_1(\deltain, e^{-\cin T}),
1+ \widetilde{B}^{(1)}_1(1+\deltaout, e^{-\cout T})\bigr)\right.\\
&\left.\qquad+ J \bigl(1+ B^{(2)}_1(1+\deltain, e^{-\cin T}), 
\widetilde{B}^{(2)}_1(\deltaout, e^{-\cout T} \bigr) = (i,j)
\right]+ \frac{1}{n}\left[\mathcal{B}^{(n)}_n= (i,j)\right]\\
=& \left(1-\frac{1}{n}\right)\PP\left[\bigl(\calI, \calO\bigr)=(i,j)\right] + \frac{1}{n}\left[\mathcal{B}^{(n)}_n= (i,j)\right].
\end{align*}
Therefore,
\[
\left|\frac{1}{n}\EE\left[N_{ij}(n)\right] - \PP\left[\bigl(\calI, \calO\bigr)=(i,j)\right]\right|\le \frac{4}{n},
\]
which leads to \eqref{eq:conv_pij} and \eqref{eq:negbin} as $n\to\infty$.}
\end{proof}
 
\begin{Remark}\label{rem:needLater}
{\rm \sid{This argument also shows that for $x>0,y>0$,
\beqq\label{eq:needLater}
\frac 1n \EE N_{>x,>y}(n)=
\PP\bigl( (\calI,\calO) \in (x,\infty]\times (y,\infty]\bigr)
+\epsilon_n(x,y),\eeqq
where
$$\sup_{x>0,y>0} |\epsilon_n(x,y) | \leq \frac 4n.$$
}}
\end{Remark}

\subsection{Convergence of the joint empirical measure.}
In this section, we investigate the convergence of the joint empirical measure:
\[
\frac{1}{k_n}\sum_{k=1}^n \epsilon_{\bigl(\Din_i(n)/b_1(n/k_n),\,\Dout_i(n)/b_2(n/k_n)\bigr)}(\cdot),
\]
with scaling functions $b_i(\cdot)$, $i=1,2$, and some intermediate sequence $k_n$ such that $k_n/n\to 0$ and $k_n\to\infty$ as $n\to\infty$.
From \eqref{eq:conv_Nij}, we have
\beqq\label{eq:emp_step1}
\frac{1}{n}\sum_{v\in [n]} \epsilon_{\bigl(\Din_v(n),\Dout_v(n)\bigr)}\bigl(\{(i,j)\}\bigr) \convp \PP\bigl((\calI,\calO) = (i,j)\bigr),\qquad n\to\infty.
\eeqq
Moreover,
\cite[Theorem~2]{resnick:samorodnitsky:towsley:davis:willis:wan:2016}
shows that the limit pair $(\calI,\calO)$ is non-standard regularly
varying, i.e. 
\beqq\label{eq:emp_step2}
\sid{n\PP\left[\left(\frac{\calI}{n^{\cin}}, \frac{\calO}{n^{\cout}}\right)\in \cdot\right]\convv\gamma V_1(\cdot) + \alpha V_2(\cdot),
\quad n\to\infty,}
\eeqq
in $M_+([0,\infty]^2\setminus\{\boldsymbol{0}\})$
and  $V_i(\cdot)$, $i=1,2$, concentrate on
$(0,\infty)^2$ with Lebesgue densities given \sid{below
in} \eqref{eq:f1} and \eqref{eq:f2}.
It is also shown in \cite{wang:resnick:2016} that the density of the limit measure is jointly regularly varying, and the relationship between the
regular variation of the limit measure and that of the limit density has been explored.

Let $b_1(t)=t^{\cin}$ and $b_2(t) = t^{\cout}$, then
heuristically, combining \eqref{eq:emp_step1} and \eqref{eq:emp_step2} gives
\begin{align}
\frac{1}{k_n}\sum_{v\in [n]} \epsilon_{\bigl(\Din_v(n)/(n/k_n)^{\cin},\,\Dout_v(n)/(n/k_n)^{\cout}\bigr)}(\cdot)
&\approx \frac{n}{k_n}\PP\left[\left(\frac{\calI}{(n/k_n)^{\cin}}, \frac{\calO}{(n/k_n)^{\cout}}\right)\in \cdot\right]\label{eq:approx_emp}\\
& \Rightarrow \gamma V_1(\cdot) + \alpha V_2(\cdot),\quad n\to\infty\nonumber
\end{align}
in $\mathbb{M}([0,\infty]^2\setminus\{\boldsymbol{0}\})$.
\sid{We} justify the approximation in \eqref{eq:approx_emp} and the convergence result is summarized in the following theorem.
\begin{Theorem}\label{thm:tail_meas}
Suppose that $\{k_n\}$ is an intermediate sequence satisfying 
\beqq\label{cond:kn}
\liminf_{n\to\infty} k_n/(n\log n)^{1/2}>0\quad \text{and}\quad
k_n/n\to 0 \quad \text{as}\quad n\to\infty,
\eeqq  
 \sid{and recall    $a=\cout/\cin$.}
Then we have
\beqq\label{eq:tailmeas}
\frac{1}{k_n}\sum_{v\in [n]}^n \epsilon_{\bigl(\Din_v(n)/(n/k_n)^{\cin},\,\Dout_v(n)/(n/k_n)^{\cout}\bigr)}(\cdot)\Rightarrow \gamma V_1(\cdot) + \alpha V_2(\cdot),
\eeqq
in $M_+([0,\infty]^2\setminus\{\boldsymbol{0}\})$, where $V_1$ and $V_2$ concentrate on $(0,\infty)^2$ with Lebesgue densities 
\begin{align}
f_1(x, y) &= \frac{x^{\deltain} y^{\deltaout-1}}{\cin\Gamma(1+\deltain)\Gamma(\deltaout)}\int_0^\infty z^{-(2+1/\cin+\deltain+a\deltaout)} e^{-x/z+y/z^{a}}\dd z,
\label{eq:f1}\\
\intertext{and}
f_2(x, y) &= \frac{x^{\deltain-1} y^{\deltaout}}{\cin\Gamma(\deltain)\Gamma(1+\deltaout)}\int_0^\infty z^{-(1+a+1/\cin+\deltain+a\deltaout)} e^{-x/z+y/z^{a}}\dd z,
\label{eq:f2}
\end{align}
respectively.
\end{Theorem}
\begin{proof}
Proving \eqref{eq:tailmeas} requires using concentration results for
\sid{degree counts $N_{i,j}(n)$ which compare counts with expected counts;
these are collected in Section \ref{sec:conc}.}
\sid{In this section we show}  for $x,y>0$,
\begin{subequations}\label{eq:ENkn}
\beqq\label{eq:ENkn1}
\left|\frac{1}{k_n}\EE\left(N_{>\left(\frac{n}{k_n}\right)^{\cin}x,\, >\left(\frac{n}{k_n}\right)^{\cout}y}(n)\right)
- \frac{n}{k_n}\, p_{>\left(\frac{n}{k_n}\right)^{\cin}x,\, >\left(\frac{n}{k_n}\right)^{\cout}y}\right| \convp 0,
\eeqq
\beqq\label{eq:ENkn2}
 \left|\frac{1}{k_n} \EE\left(\Nin_{>\left(\frac{n}{k_n}\right)^{\cin}x}(n)\right) - 
\frac{n}{k_n}\, \pin_{>\left(\frac{n}{k_n}\right)^{\cin}x}\right| \convp 0, 
\eeqq
\beqq\label{eq:ENkn3}
  \left|\frac{1}{k_n}\EE\left(\Nout_{>\left(\frac{n}{k_n}\right)^{\cout}y}(n)\right) - 
\frac{n}{k_n}\, \pout_{>\left(\frac{n}{k_n}\right)^{\cout}y}\right| \convp 0.
\eeqq
\end{subequations}
We give a proof for \eqref{eq:ENkn1} and \eqref{eq:ENkn2} and \eqref{eq:ENkn3}
follow\sid{s} from a similar argument.

Adopting the notation from the proof of Theorem~\ref{thm:conv_Nij},
\sid{and using \eqref{eq:needLater}} we have
\begin{align*}
\Bigl|\frac{1}{k_n}\EE &
  \left(N_{>\left(\frac{n}{k_n}\right)^{\cin}x,\,
  >\left(\frac{n}{k_n}\right)^{\cout}y}(n)\right) 
- \frac{n}{k_n}\, p_{>\left(\frac{n}{k_n}\right)^{\cin}x,\, >\left(\frac{n}{k_n}\right)^{\cout}y}\Bigr|\\
= & \left| \frac{n}{k_n}\frac{1}{n}\sum_{v\in [n]}\PP\left(\frac{\Din_v(n)}{(n/k_n)^{\cin}}> x,\, \frac{\Dout_v(n)}{(n/k_n)^{\cout}}> y\right)
- \frac{n}{k_n}\PP\left[\frac{\calI}{(n/k_n)^{\cin}}>x, \frac{\calO}{(n/k_n)^{\cout}}>y\right]\right|\\
= & \left|\frac{n}{k_n}\frac{1}{n}\sum_{v=1}^n\PP\left(\mathcal{B}^{(n)}_v\in \bigl((n/k_n)^{\cin}x, \infty\bigr]\times \bigl((n/k_n)^{\cout}y, \infty\bigr]\right)\right.\\
&\left.\qquad - \frac{n}{k_n}\PP\left[\frac{\calI}{(n/k_n)^{\cin}}>x, \frac{\calO}{(n/k_n)^{\cout}}>y\right]\right|\\
& + \frac{1}{k_n}\left|\PP\left(\frac{\Din_1(n)}{(n/k_n)^{\cin}}> x,\,
  \frac{\Dout_1(n)}{(n/k_n)^{\cout}}> y\right) \right.\\
&\left.\qquad \qquad
- \PP\left(\mathcal{B}^{(n)}_1\in
  \bigl((n/k_n)^{\cin}x, \infty\bigr]\times \bigl((n/k_n)^{\cout}y,
  \infty\bigr]\right)\right| \\
&\le \sid{\epsilon_n\bigl((n/k_n)^{\cin}x,(n/k_n)^{\cout}y \bigr)           + \frac{2}{k_n}\to 0,}
\end{align*}
as $n\to\infty$.

Combining concentration results in \eqref{eq:con_2d}, \eqref{eq:marg_in} and \eqref{eq:marg_out} with \eqref{eq:ENkn}
implies that for \sid{any} intermediate sequence $\{k_n\}$ satisfying \eqref{cond:kn} and $x, y> 0$, as $n\to\infty$,
\begin{subequations}\label{eq:Nkn}
\beqq
\frac{1}{k_n} \left|N_{>\left(\frac{n}{k_n}\right)^{\cin}x,\, >\left(\frac{n}{k_n}\right)^{\cout}y}(n) - 
n\, p_{>\left(\frac{n}{k_n}\right)^{\cin}x,\, >\left(\frac{n}{k_n}\right)^{\cout}y}\right| \convp 0, 
\eeqq
\beqq\label{eq:Nkn2}
\frac{1}{k_n}  \left|\Nin_{>\left(\frac{n}{k_n}\right)^{\cin}x}(n) -
n\, \pin_{>\left(\frac{n}{k_n}\right)^{\cin}x}\right| \convp 0, 
\eeqq
\beqq
\frac{1}{k_n}  \left|\Nout_{>\left(\frac{n}{k_n}\right)^{\cout}y}(n) - 
n\, \pout_{>\left(\frac{n}{k_n}\right)^{\cout}y}\right| \convp 0.
\eeqq
\end{subequations}
\tw{Define the vague metric $\rho(\cdot,\cdot)$ on $M_+([0,\infty]^2\setminus\{\origin\})$ (cf. \cite[Chapter 3.3]{resnickbook:2007}) 
as follows. There exists some sequence of 
continuous functions on $[0,\infty]^2\setminus\{\origin\}$ with compact supports, $f_i: [0,\infty]^2\setminus\{\origin\}\mapsto \RR_+$, $i\ge 1$,
and for $\mu_1,\mu_2\in M_+([0,\infty]^2\setminus\{\origin\})$,
\[
\rho(\mu_1,\mu_2) = \sum_{i=1}^\infty \frac{|\mu_1(f_i)-\mu_2(f_i)|\min 1}{2^i},
\]
where $\mu_j(f_i) := \int_{[0,\infty]^2\setminus\{\origin\}}f_i(x)\mu_j(\dd x)$, $j=1,2$, $i\ge 1$.
Then results in \eqref{eq:Nkn} imply:
as $n\to\infty$,
\beqq\label{eq:approx}
\rho\left(\frac{1}{k_n}\sum_{v\in [n]}^n \epsilon_{\bigl(\Din_v(n)/(n/k_n)^{\cin},\,\Dout_v(n)/(n/k_n)^{\cout}\bigr)}, \frac{n}{k_n}\PP\left[\left(\frac{\calI}{(n/k_n)^{\cin}}, \frac{\calO}{(n/k_n)^{\cout}}\right)\in\cdot\right]\right)\convp 0.
\eeqq}
Then \eqref{eq:tailmeas} follows from combining \eqref{eq:approx} and the vague convergence in \eqref{eq:emp_step2}, with \eqref{eq:f1} and \eqref{eq:f2} being specified in \cite[Theorem~2]{resnick:samorodnitsky:towsley:davis:willis:wan:2016}.
\end{proof}

\section{Consistency of the Hill Estimator}\label{sec:Hill}
In practice, the growth rates of in- and out-degrees are often estimated by Hill estimators as defined in \eqref{eq:hill_def}. However, despite its wide use, 
there is no theoretical justification for such estimates and the consistency has been proved only for a simple undirected preferential attachment model in \cite{wang:resnick:2017}. We now turn to \eqref{eq:conv_meas_in} and \eqref{eq:conv_meas_out} as preparations for considering consistency of the Hill estimator.

\begin{Proposition}\label{prop:conv_meas_marg}
Suppose that $\{k_n\}$ is some intermediate sequence satisfying 
\eqref{cond:kn}. Define
\begin{align*}
b_1(t) &= \left[\cin\frac{\Gamma(1+\deltain + \cin^{-1})}{\Gamma(1+\deltain)}\left(\frac{\alpha\deltain}{1+\cin\deltain}+\frac{\gamma}{\cin}\right)\right]^{\cin} t^{\cin},\\
b_2(t) &= \left[\cout\frac{\Gamma(1+\deltaout + \cout^{-1})}{\Gamma(1+\deltaout)}\left(\frac{\gamma\deltaout}{1+\cout\deltaout}+\frac{\alpha}{\cout}\right)\right]^{\cout} t^{\cout},
\end{align*}
then 
\begin{align}
\frac{1}{k_n}\sum_{v\in [n]}\epsilon_{\Din_v(n)/b_1(n/k_n)} &\convw \nu_{\cin^{-1}},\qquad\text{in }M_+((0,\infty]),\label{eq:conv_min}\\
\frac{1}{k_n}\sum_{v\in [n]}\epsilon_{\Dout_v(n)/b_2(n/k_n)} &\convw \nu_{\cout^{-1}}, \qquad\text{in }M_+((0,\infty]).\label{eq:conv_mout}
\end{align}
\end{Proposition}
\begin{proof}
Marginalizing the results in \eqref{eq:tailmeas} gives 
\begin{align*}
\frac{1}{k_n}\sum_{v\in [n]}\epsilon_{\frac{\Din_v(n)}{(n/k_n)^{\cin}}} &\convw \cin\frac{\Gamma(1+\deltain + \cin^{-1})}{\Gamma(1+\deltain)}\left(\frac{\alpha\deltain}{1+\cin\deltain}+\frac{\gamma}{\cin}\right)\nu_{\cin^{-1}},\qquad\text{in }M_+((0,\infty]),\\
\frac{1}{k_n}\sum_{v\in [n]}\epsilon_{\frac{\Dout_v(n)}{(n/k_n)^{\cout}}} &\convw \cout\frac{\Gamma(1+\deltaout + \cout^{-1})}{\Gamma(1+\deltaout)}\left(\frac{\gamma\deltaout}{1+\cout\deltaout}+\frac{\alpha}{\cout}\right)\nu_{\cout^{-1}},\qquad\text{in }M_+((0,\infty]).
\end{align*} 
Scaling both sides by the constant appearing in the limit measure gives \eqref{eq:conv_min} and \eqref{eq:conv_mout}.
\end{proof}

With Proposition~\ref{prop:conv_meas_marg} available, we now prove the consistency of Hill estimators for in- and out-degrees.
\begin{Theorem}\label{thm:hill}
Let 
\begin{align*}
&\Din_{(1)}(n)\ge \Din_{(2)}(n) \ge\cdots\ge \Din_{(n)}(n),\\
&\Dout_{(1)}(n)\ge \Dout_{(2)}(n) \ge\cdots\ge \Dout_{(n)}(n),
\end{align*}
be order statistics for in- and out-degrees $\{\Din_v(n)\}_{v\in [n]}$, $\{\Dout_v(n)\}_{v\in [n]}$, respectively.
Define the Hill estimators for $\{\Din_v(n)\}_{v\in [n]}$ and $\{\Dout_v(n)\}_{v\in [n]}$ as
\[
H^\text{in}_{k,n} := \frac{1}{k}\sum_{i=1}^{k}\log\frac{\Din_{(i)}(n)}{\Din_{(k+1)}(n)},\qquad 
H^\text{out}_{k,n} := \frac{1}{k}\sum_{i=1}^{k}\log\frac{\Dout_{(i)}(n)}{\Dout_{(k+1)}(n)}.
\]
Then for some intermediate sequence $\{k_n\}$ satisfying 
\eqref{cond:kn}, we have as $n\to\infty$,
\beqq\label{eq:hill_cst}
H^\text{in}_{k_n,n}\convp \cin,\qquad\text{and}\qquad H^\text{out}_{k_n,n}\convp \cout.
\eeqq
\end{Theorem}
\begin{proof}
From \eqref{eq:conv_min} and \eqref{eq:conv_mout}, we conclude by inversion and \cite[Proposition~3.2]{resnickbook:2007}
that in $D(0,\infty]$
\[
\frac{\Din_{([k_n t])}(n)}{b_1(n/k_n)}\convp t^{-\cin}\quad \text{and}\quad \frac{\Dout_{([k_n t])}(n)}{b_2(n/k_n)}\convp t^{-\cout}.
\]
Therefore,
\begin{align}
\left(\frac{1}{k_n}\sum_{v\in [n]}\epsilon_{\Din_v(n)/b_1(n/k_n)}, \frac{\Din_{(k_n)}(n)}{b_1(n/k_n)}\right) &\convw \bigl(\nu_{\cin^{-1}}, 1\bigr)
\quad \text{in } M_+((0,\infty])\times (0,\infty),\label{eq:hill_in}\\
\left(\frac{1}{k_n}\sum_{v\in [n]}\epsilon_{\Dout_v(n)/b_2(n/k_n)}, \frac{\Dout_{(k_n)}(n)}{b_2(n/k_n)}\right) &\convw \bigl(\nu_{\cout^{-1}}, 1\bigr)
\quad \text{in } M_+((0,\infty])\times (0,\infty).\label{eq:hill_out}
\end{align}

Define the operator
\[
S: M_+((0,\infty])\times (0,\infty) \mapsto M_+((0,\infty])
\]
by 
\[
S(\nu, c)(A) = \nu(cA).
\]
By the proof in \cite[Theorem~4.2]{resnickbook:2007}, the mapping $S$ is continuous at $(\nu_{c_i^{-1}},1)$, $i=1,2$.
Therefore, applying the continuous mapping $S$ to the joint weak convergence in \eqref{eq:hill_in} and \eqref{eq:hill_out} gives 
\begin{align*}
\frac{1}{k_n}\sum_{v\in [n]}\epsilon_{\Din_v(n)\big/\Din_{(k_n)}(n)} &\convw \nu_{\cin^{-1}},\qquad\text{in }M_+((0,\infty]),\\
\frac{1}{k_n}\sum_{v\in [n]}\epsilon_{\Dout_v(n)\big/\Dout_{(k_n)}(n)} &\convw \nu_{\cout^{-1}}, \qquad\text{in }M_+((0,\infty]).
\end{align*}
Then the rest of the proof is similar to arguments in the proof of \cite[Theorem~11]{wang:resnick:2017}. Here we only include proofs for
 the consistency $H^\text{in}_{k_n,n}$ and that for $H^\text{out}_{k_n,n}$ follows from the same argument.
 Define $\hat{\nu}_n^\text{in}(\cdot) := \frac{1}{k_n}\sum_{v\in [n]}\epsilon_{\Din_v(n)\big/\Din_{(k_n)}(n)}(\cdot)$.
 First observe
\[
H^\text{in}_{k_n, n} = \int_1^\infty \hat{\nu}_n^\text{in}(y,\infty] \frac{\dd y}{y}.
\]
Then fix $M>0$ large and 
define a mapping $f\mapsto \int_1^M f(y)\frac{\dd y}{y}$ from
$D(0,\infty] \mapsto \RR_+$. This map is a.s. continuous so
\[
\int_1^M \hat{\nu}_n^\text{in}(y,\infty] \frac{\dd y}{y}\convp \int_1^M \nu_{\cin^{-1}}(y,\infty] \frac{\dd y}{y},
\]
and it remains to show 
by the second converging together theorem 
(cf. \cite[Theorem~3.5]{resnickbook:2007}) that
\beqq\label{eq:cont_integral}
\lim_{M\to\infty} \limsup_{n\to\infty} \PP\left(\int_M^\infty \hat{\nu}_n^\text{in}(y,\infty] \frac{\dd y}{y}>\varepsilon\right) = 0.
\eeqq
The probability in \eqref{eq:cont_integral} is
\begin{align*}
\PP & \left(\int_M^\infty \hat{\nu}_n^\text{in}(y,\infty] \frac{\dd y}{y}>\varepsilon\right)
\le\, \PP\left(\int_M^\infty \hat{\nu}_n^\text{in}(y,\infty] \frac{\dd y}{y}>\varepsilon, \left|\frac{\Din_{(k_n)}(n)}{b_1(n/k_n)}-1\right|<\eta\right)\\
&+ \PP\left(\int_M^\infty \hat{\nu}_n^\text{in}(y,\infty] \frac{\dd y}{y}>\varepsilon, \left|\frac{\Din_{(k_n)}(n)}{b_1(n/k_n)}-1\right|\ge\eta\right)\\
\le\, \PP & \left(\int_M^\infty \frac{1}{k_n}\sum_{i=1}^n \epsilon_{\Din_i(n)/b_1(n/k_n)}((1-\eta)y,\infty] \frac{\dd y}{y}>\varepsilon\right)\\
 +& \PP\left(\left|\frac{\Din_{(k_n)}(n)}{b_1(n/k_n)}-1\right|\ge\eta\right) =: A+B.
\end{align*}
By \eqref{eq:hill_in}, $B\to 0$ as $n\to\infty$, and using the Markov inequality, $A$ is bounded by
\begin{align*}
&\frac{1}{\varepsilon} \EE\left(\int_M^\infty \frac{1}{k_n}\sum_{v=1}^n \epsilon_{\Din_v(n)/b_1(n/k_n)}((1-\eta)y,\infty] \frac{\dd y}{y}\right)\\
=& \frac{1}{\varepsilon}\EE\left(\int_{M(1-\eta)}^\infty \frac{1}{k_n}\sum_{v=1}^n \epsilon_{\Din_v(n)/b_1(n/k_n)}(y,\infty] \frac{\dd y}{y}\right)
\le \frac{1}{\varepsilon}\int_{M(1-\eta)}^\infty \frac{1}{k_n}\EE\left(\Nin_{>[b_1(n/k_n) y]}(n)\right) \frac{\dd y}{y}.
\end{align*}
Using Stirling's formula, \eqref{eq:Nkn2} gives that for $y>0$,
\beqq\label{eq:conv_Eoverk}
\frac{1}{k_n}\EE\left(\Nin_{>[b_1(n/k_n) y]}(n)\right)\rightarrow y^{-\cin^{-1}}.
\eeqq
Let $U(t) := \EE(\Nin_{>t}(n))$ and \eqref{eq:conv_Eoverk} becomes: for $y>0$,
\[
\frac{1}{k_n} U(b_1(n/k_n) y) \rightarrow y^{-\cin^{-1}},\quad \text{as }n\to\infty.
\]
Since $U(\cdot)$ is a non-increasing function, $U\in RV_{-\cin^{-1}}$ by \cite[Proposition 2.3(ii)]{resnickbook:2007}.
Therefore, Karamata's theorem gives
\[
A \le \frac{1}{\varepsilon}\int_{M(1-\eta)}^\infty \frac{1}{k_n}\EE\left(\Nin_{>[b_1(n/k_n) y]}(n)\right) \frac{\dd y}{y}
\sim C(\delta, \eta) M^{-\cin^{-1}},
\]
with some positive constant $C(\delta, \eta)>0$.
Also, $M^{-\cin^{-1}}\to 0$ as $M\to\infty$, and \eqref{eq:cont_integral} follows.
\end{proof}

\section{Concentration of degree counts}\label{sec:conc}
In this section, we collect concentration results for the degree counts that are useful in the proofs in Theorem~\ref{thm:tail_meas}.
\begin{Lemma}\label{lemma:concen}
Define $N_{>i, >j}(n) := \sum_{v\in [n]} \ind_{\left\{\Din_v(n)>i, \Dout_v(n)>i\right\}}$.
Then for $\deltain>0$, there exists a constant $C>6$ such that as $n\to\infty$,
\beqq\label{eq:con_2d}
\PP\left(\max_{i,j} \left|N_{>i, >j}(n) - \EE(N_{>i, >j}(n))\right|\ge C\bigl(1+ \sqrt{n\log n}\bigr)\right) = o(1).
\eeqq
\end{Lemma}
\begin{proof}
The proof of \eqref{eq:con_2d} follows from a similar argument as in the proof of \cite[Proposition 8.4]{vanderHofstad:2017}. 
We include it here to make it self-contained. Define a martingale 
\[
M_m := \EE\bigl(N_{>i,>j}(n)|G(m)\bigr) = \sum_{v\in [n]} \PP\left(\Din_v(n)>i, \Dout_v(n)>j|G(m)\right),\quad m\le n.
\]
For $m\ge 2$, we define a new graph $G'(s)$ by $G'(s)=G(s)$ for $s\le m-1$, while $s\mapsto G'(m)$ evolves independently of $\{G(s): s\ge m-1\}$,
following the preferential attachment rule given in Section~\ref{subsec:PA}.
Denote the in- and out-degrees of the node $v$ in $G'(n)$ by $(\Din)'_v(n), (\Dout)'_v(n)$,
we then have
\beqq\label{eq:Mm}
M_{m-1} = \sum_{v\in [n]} \PP\left((\Din)'_v(n)>i, (\Dout)'_v(n)>j|G(m-1)\right).
\eeqq
Since the evolution of $s\mapsto G'(s)$ is independent of that of $\{G(s): s\ge m-1\}$ for $s\ge m-1$, it makes no difference whether we condition on $G(m-1)$ or $G(m)$ in \eqref{eq:Mm}. Hence, we have
\begin{align}
M_m- &M_{m-1} \label{eq:M_diffa}\\
&= \sum_{v\in [n]}\bigl[\PP\bigl(\Din_v(n)>i, \Dout_v(n)>j\big|G(m)\bigr) - \PP\bigl((\Din)'_v(n)>i, (\Dout)'_v(n)>j\big|G(m)\bigr)\bigr].\nonumber
\end{align}

Since the evolution of $n\mapsto \bigl( \Din_v(n),\Dout_v(n) \bigr)$ for $n\ge m$ only depends on $\bigl( \Din_v(m),\Dout_v(m) \bigr)$, then
\begin{align*}
&\PP\bigl(\Din_v(n)>i, \Dout_v(n)>j\big|G(m)\bigr) = \PP\bigl(\Din_v(n)>i, \Dout_v(n)>j\big|\bigl( \Din_v(m),\Dout_v(m) \bigr)\bigr),\\
&\PP\bigl((\Din)'_v(n)>i, (\Dout)'_v(n)>j\big|G(m)\bigr) \\
&= \EE\Bigl\{   \PP\bigl[(\Din)'_v(n)>i, (\Dout)'_v(n)>j\big|\bigl( (\Din)'_v(m),(\Dout)'_v(m) \bigr)\bigr] \Big| G(m) \Bigr\}.
\end{align*}
Then \eqref{eq:M_diffa} becomes
\begin{align}
M_m- M_{m-1} & \label{eq:M_diffb}\\
= \sum_{v\in [n]}\EE & \Bigl\{  \PP\bigl[\Din_v(n)>i, \Dout_v(n)>j\big|\bigl( \Din_v(m),\Dout_v(m) \bigr)\bigr] \nonumber\\
- & \PP\bigl[(\Din)'_v(n)>i, (\Dout)'_v(n)>j\big|\bigl( (\Din)'_v(m),(\Dout)'_v(m) \bigr)\bigr] \Big| G(m) \Bigr\}.\nonumber
\end{align}
It is important to note that 
\begin{align*}
\PP\bigl[\Din_v(n)>i, & \Dout_v(n)>j\big|\bigl( \Din_v(m),\Dout_v(m) \bigr)\bigr] \\
&= \PP\bigl[(\Din)'_v(n)>i, (\Dout)'_v(n)>j\big|\bigl( (\Din)'_v(m),(\Dout)'_v(m) \bigr)\bigr],
\end{align*}
as long as $\bigl( \Din_v(m),\Dout_v(m) \bigr) = \bigl( (\Din)'_v(m),(\Dout)'_v(m) \bigr)$, because the two graphs are constructed based on the same 
preferential attachment rule. Thus,
\begin{align*}
\Big|\PP\bigl[\Din_v(n)>i, & \Dout_v(n)>j\big|\bigl( \Din_v(m),\Dout_v(m) \bigr)\bigr] \\
&- \PP\bigl[(\Din)'_v(n)>i, (\Dout)'_v(n)>j\big|\bigl( (\Din)'_v(m),(\Dout)'_v(m) \bigr)\bigr]\Big|\\
\le &\ind_{\left\{\bigl( \Din_v(m),\Dout_v(m) \bigr) \neq \bigl( (\Din)'_v(m),(\Dout)'_v(m) \bigr)\right\}}.
\end{align*}
So we conclude that \eqref{eq:M_diffb} is bounded by:
\begin{align*}
|M_m &- M_{m-1} |\\
\le \sum_{v\in [n]} \EE & \Bigl\{  \big|\PP\bigl[\Din_v(n)>i, \Dout_v(n)>j\big|\bigl( \Din_v(m),\Dout_v(m) \bigr)\bigr]\\
- & \PP\bigl[(\Din)'_v(n)>i, (\Dout)'_v(n)>j\big|\bigl( (\Din)'_v(m),(\Dout)'_v(m) \bigr)\bigr]\big| \Big| G(m) \Bigr\}\\
\le \sum_{v\in [n]} \EE & \left(\ind_{\left\{\bigl( \Din_v(m),\Dout_v(m) \bigr) \neq \bigl( (\Din)'_v(m),(\Dout)'_v(m) \bigr)\right\}}\middle|G(m)\right)\\
= \EE & \left(\sum_{v\in [n]} \ind_{\left\{\bigl( \Din_v(m),\Dout_v(m) \bigr) \neq \bigl( (\Din)'_v(m),(\Dout)'_v(m) \bigr)\right\}}\middle|G(m)\right).
\end{align*}
Note that $\bigl( \Din_v(m-1),\Dout_v(m-1) \bigr) \neq \bigl( (\Din)'_v(m-1),(\Dout)'_v(m-1) \bigr)$ for all $1\le v\le m-1$ by construction, and since
changing an edge will change the in- and out-degrees for at most 3 nodes, then
\[
|M_m - M_{m-1} | \le 3.
\]

Next, we use the Azuma-Hoeffding inequality to prove \eqref{eq:con_2d}. Since $N_{>i,>j}(n) = 0$ for $i,j>n$, then
\begin{align*}
\PP & \left(\max_{i,j} \left|N_{>i, >j}(n) - \EE(N_{>i, >j}(n))\right|\ge C \sqrt{n\log n}\right)\\
& \le \sum_{i=0}^{n-1} \sum_{j=0}^{n-1} \PP \left(\left|N_{>i, >j}(n) - \EE(N_{>i, >j}(n))\right|\ge C \sqrt{n\log n}\right)\\
&\le n^2\cdot 2\exp\left\{-\frac{C^2 \log n}{2\cdot 3^2}\right\} = 2n^{-(C^2/18 - 2)}.
\end{align*}
Therefore,  \eqref{eq:con_2d} follows from taking $C>6$.
\end{proof}

Results in Lemma~\ref{lem:marg_concen} also follows from the argument in \cite[Proposition 8.4]{vanderHofstad:2017} Since the details of this proof machinery has been given in the proof of Lemma~\ref{lemma:concen}, they are omitted here.
\begin{Lemma}\label{lem:marg_concen}
For $\deltain, \deltaout >0 $, there exist constants $C_\text{in}, C_\text{out}>3\sqrt{2}$, such that as $n\to\infty$,
\begin{align}
\PP & \left(\max_{i\ge 0}\bigl|\Nin_{>i}(n) -  \EE(\Nin_{>i}(n))\bigr|\ge C_\text{in}(1+\sqrt{n\log n})\right) = o(1),\label{eq:marg_in}\\
\intertext{and}
\PP & \left(\max_{j\ge 0}\bigl|\Nout_{>j}(n) -  \EE(\Nout_{>i}(n))\bigr|\ge C_\text{out}(1+\sqrt{n\log n})\right) = o(1).\label{eq:marg_out}
\end{align}
\end{Lemma}

\bibliography{./bibfile}
\end{document}